\documentclass[ssy]{imsart}
\usepackage{graphicx,color}
\usepackage{amsmath,amsfonts,amssymb,amsthm}
\usepackage[colorlinks]{hyperref}
\usepackage{ulem}
\usepackage{enumerate}
\usepackage{algorithm}
\usepackage{algpseudocode}

\usepackage[numbers]{natbib}
\startlocaldefs

\newcommand{\blue}[1]{{#1}}
\newcommand{\red}[1]{}
\definecolor{dgreen}{rgb}{0.05, 0.6, 0.2}
\newcommand{\dgreen}[1]{{\color{dgreen} #1}}

\newcommand{\rs}[1]{{\mbox{\scriptsize \sc #1}}}
\newcommand{\vc}[1]{{\boldsymbol #1}}

\newcommand{\sr}[1]{{\cal #1}}
\newcommand{\dd}[1]{\mathbb{#1}}
\newcommand{\E}{\mathbb{E}}
\newcommand{\R}{\mathbb{R}}
\newcommand{\Prob}{\mathbb{P}}
\newcommand{\Z}{\mathbb{Z}}
\newcommand{\J}{\mathcal{J}}
\providecommand{\abs}[1]{\lvert#1\rvert}

\newcommand{\eq}[1]{(\ref{eq:#1})}
\newcommand{\lem}[1]{Lemma~\ref{lem:#1}}

\newcommand{\thr}[1]{Theorem~\ref{thr:#1}}

\newcommand{\app}[1]{Section~\ref{app:#1}}
\newcommand{\sectn}[1]{Section~\ref{sect:#1}}

\newcommand{\lemt}[1]{\ref{lem:#1}}

\newcommand{\appt}[1]{\ref{app:#1}}

\newcommand{\br}[1]{\langle #1 \rangle}

\newcommand{\pend}{\hfill \thicklines \framebox(6.6,6.6)[l]{}}

\newenvironment{proof*}[1]{\noindent {\sc  #1} \rm}{\pend}

\newtheorem{theorem}{Theorem}[section]
\newtheorem{lemma}{Lemma}[section]

\newtheorem{proposition}{Proposition}[section]

\newtheorem{remark}{Remark}[section]

\newcommand{\setsection}[2] {
\setcounter{section}{#1}
\setcounter{subsection}{0}
\setcounter{equation}{0}
\setcounter{conjecture}{0}
\setcounter{assumption}{0}
\setcounter{question}{0}
\setcounter{definition}{0}
\setcounter{theorem}{0}
\setcounter{corollary}{0}
\setcounter{lemma}{0}
\setcounter{proposition}{0}
\setcounter{remark}{0}
\setcounter{appen}{0}
\setsection*{\large \bf \thesection. #2}}

\newcommand{\chapternote}[1]{{%
  \let\thempfn\relax
  \footnotetext[0]{#1}
}}

\newcommand{\setnewcounter} {
\setcounter{subsection}{0}
\setcounter{equation}{0}
\setcounter{conjecture}{0}
\setcounter{assumption}{0}
\setcounter{question}{0}
\setcounter{definition}{0}
\setcounter{theorem}{0}
\setcounter{corollary}{0}
\setcounter{lemma}{0}
\setcounter{proposition}{0}
\setcounter{remark}{0}
}
\endlocaldefs

\begin{document}
\begin{frontmatter}

\title{Heavy traffic approximation for the stationary distribution of a generalized Jackson network: the BAR approach}
\runtitle{ Steady-state heavy traffic approximation: BAR approach}


\author{\fnms{Anton} \snm{Braverman}\thanksref{t1,t2,m1}\corref{}\ead[label=e1]{ab2329@cornell.edu}},
\author{\fnms{J.G.} \snm{Dai}\thanksref{t3,m1}\ead[label=e2]{jd694@cornell.edu}}
\and
\author{\fnms{Masakiyo} \snm{Miyazawa}\thanksref{t1,m2}\ead[label=e3]{miyazawa@rs.tus.ac.jp}}
\address{\printead{e1}}
\address{\printead{e2}}
\address{\printead{e3}}
\affiliation{Cornell University\thanksmark{m1} and Tokyo University of Science\thanksmark{m2}}

\runauthor{Braverman, Dai, and Miyazawa}

\begin{abstract}
In the seminal paper of \citet{GamaZeev2006}, the authors justify the steady-state diffusion approximation of a generalized Jackson network (GJN) in heavy traffic. Their approach involves the so-called limit interchange argument, which has since become a popular tool employed by many others who study diffusion approximations. In this paper we illustrate a novel approach by using it to justify the steady-state approximation of a GJN in heavy traffic. Our approach involves working directly with the basic adjoint relationship (BAR), an integral equation that characterizes the stationary distribution of a Markov process. As we will show, the BAR approach is a more natural choice than the limit interchange approach for justifying steady-state approximations, and can potentially be applied to the study of other stochastic processing networks such as multiclass queueing networks. 
\end{abstract}



\end{frontmatter}




\begin{abstract}
In the seminal paper of \citet{GamaZeev2006}, the authors justify the steady-state diffusion approximation of a generalized Jackson network (GJN) in heavy traffic. Their approach involves the so-called limit interchange argument, which has since become a popular tool employed by many others who study diffusion approximations. In this paper we illustrate a novel approach by using it to justify the steady-state approximation of a GJN in heavy traffic. Our approach involves working directly with the basic adjoint relationship (BAR), an integral equation that characterizes the stationary distribution of a Markov process. As we will show, the BAR approach is a more natural choice than the limit interchange approach for justifying steady-state approximations, and can potentially be applied to the study of other stochastic processing networks such as multiclass queueing networks. 
\end{abstract}
  

\section{Introduction}
\label{sect:introduction} 
\chapternote{\dgreen{This document serves as an improved version of the recently published paper \cite{BravDaiMiya2017}.
All changes are colored in green, and are documented in the last paragraph of the introduction, just before Section~\ref{sect:notation}.}}
This paper considers open single-class queueing networks that have $d$ service stations. Each station has a single server operating under the first-in-first-out (FIFO) service discipline. Upon completing service at a particular station, customers are either routed to another station, or exit the network. There is a single class of customers at each station, meaning that all customers are homogenous in terms of service times and routing. A customer entering the network will exit in finite time with probability one, hence the term open network. For each station, the external interarrival times (possibly
null), service times, and routing decisions are assumed to follow three separate i.i.d.\
sequences of random variables; the three sequences are assumed to be independent. Furthermore, the interarrival times, service times and routing decisions are assumed to be
independent between different stations. Such a network is hereafter referred to as a generalized
Jackson network (GJN).

 In a seminal paper, Gamarnik and Zeevi~\cite{GamaZeev2006} proved that for a
sequence of GJNs indexed by $n=1, 2, \ldots$,
\begin{equation}
  \label{eq:steadyConver}
  r_n L ^{(n)}(\infty) \Rightarrow Z(\infty) \quad \text{ as } n\to\infty,
\end{equation}
where the symbol $\Rightarrow$ denotes convergence in distribution,
$\{r_n\}$ is a sequence of positive numbers that converge to zero, $L^{(n)}(\infty)$ is a
random vector whose $i$th component represents the steady-state number
of customers at station $i$ in the $n$th network, and $Z(\infty)$ is a
random vector that has the stationary distribution of a certain
$d$-dimensional semimartingale reflecting Brownian motion (SRBM)
$Z=\{Z(t), t\ge 0\}$ that was first defined in \cite{HarrReim1981}.
Readers are referred to the introduction of \cite{GamaZeev2006} for
the importat motivation of this problem, and a review of then recent
literature.
  Gamarnik and Zeevi \cite{GamaZeev2006} proved (\ref{eq:steadyConver}) under
two key conditions: (a) the heavy traffic condition, and (b) the
exponential moment condition. 

Condition (a) is standard and can be
expressed in terms of a $d$-dimensional vector $\rho^{(n)}$, where $\rho^{(n)}_i$
is the traffic intensity at station $i$ in the $n$th network. This
condition requires that $\rho^{(n)}_i<1$ at each station $i$, and
$\rho^{(n)}_i\to 1$ as $n\to\infty$. The scaling parameter $r_n$ in
(\ref{eq:steadyConver}) is closely tied to this heavy traffic condition and describes how quickly each $\rho^{(n)}_i$ converges to one. In particular, $r_n$ goes to zero at the same rate as $1-\rho^{(n)}_i$. Namely, $\lim_{n \to \infty} (1-\rho^{(n)}_i)/r_n > 0$ exists for each station $i$.
 Condition (b) requires that
interarrival and service times have finite exponential moments; such a condition is unnecessarily strong. In a follow up work by Budhiraja
and Lee \cite{BudhLee2009}, this condition is relaxed to a new moment
condition: (b') interarrival and service times have finite second
moments, and the sequences (indexed by $n$) of square interarrival and square service times are uniformly integrable.  Conditions (a) and (b') in \cite{BudhLee2009} represent the
\blue{weakest possible} conditions for (\ref{eq:steadyConver}) to hold.

In this paper, we prove (\ref{eq:steadyConver}) under conditions (a) and (b'); see Theorem~\ref{thr:heavy traffic 1} in Section~\ref{sect:results}. Our proof of (\ref{eq:steadyConver}) uses a novel approach, and is drastically different from the ones in
\cite{BudhLee2009} and \cite{GamaZeev2006}. This approach was used in \cite{Miya2015} to study the steady-state approximation of a single server queue in heavy traffic. However, ours is the first paper to apply it to the \red{\sout{GJN}} \blue{network} setting. In addition to proving Theorem
\ref{thr:heavy traffic 1}, the ideas laid out in this paper can be applied to study steady-state diffusion approximations of other systems of interest. One promising direction for future work is to generalize this approach to study multiclass queuing networks, such as those studied by \citet{BramDai2001}. We now outline the approach.

For $\theta\in \R^d$ with $\theta\le 0$, let $\varphi^{(n)}(\theta)$ be the moment generating function (MGF) of
$Z^{(n)}(\infty) =r_nL^{(n)}(\infty)$, defined in
(\ref{eq:laplaceprelimit}). To prove (\ref{eq:steadyConver}), we show that $\varphi^{(n)}(\theta)$ converges to $\varphi(\theta)$,  the MGF of
$Z(\infty)$. To do so, it suffices to prove that the pointwise limit of any convergent
subsequence $\{\varphi^{(n_k)},\ k \geq 1\}$ must be $\varphi(\theta)$, i.e.
\begin{equation}
  \label{eq:mgfsubconv}
  \lim_{k\to\infty} \varphi^{(n_k)}(\theta) = \varphi(\theta) \text{ for each } \theta\in \R^d \text{ with } \theta\le 0.
\end{equation}
Also associated with $Z(\infty)$ are boundary MGFs $\varphi_j(\theta)$ $(j=1, \ldots, d)$, defined in \eqref{eq:mgfsrbmdef}. By a uniqueness result in \cite{DaiKurt1994}, we know that $\varphi(\theta)$ and its boundary counterparts are characterized by a basic
adjoint relationship (BAR) given in (\ref{eq:srbmbar}). We know that $\varphi^{(n_k)}(\theta)$ also has associated boundary MGFs $\varphi^{(n_k)}_j(\theta)$ $(j=1, \ldots, d)$ that are defined in \eqref{eq:laplaceprelimit}. To prove
(\ref{eq:mgfsubconv}), we show in Proposition~\ref{lem:prelimitMGFBAR} that $\varphi^{(n_k)}(\theta)$ and its boundary counterparts satisfy BAR (\ref{eq:srbmbar}) asymptotically. Namely, the limits $\varphi^*(\theta)=\lim_{k\to\infty}\varphi^{(n_k)}(\theta)$,
along with
$\varphi^*_j(\theta)=\lim_{k\to\infty}\varphi^{(n_k)}_j(\theta)$, 
satisfy BAR (\ref{eq:srbmbar}) exactly. On its own, this result does not yet imply that $\varphi^*(\theta)=\varphi(\theta)$. 

\blue{Apriori, we cannot exclude the possibility that} $\varphi^*(\theta)$ may be degenerate, i.e.\ be the MGF of some nonnegative measure on $\R^d$ that has total mass strictly
less than $1$ (or possibly $0$). To invoke the uniquness result in \cite{DaiKurt1994} and
conclude that $\varphi^*(\theta)=\varphi(\theta)$, we must also  prove that $\varphi^*(\theta)$ is the MGF of \textit{some} probability
measure, i.e.\ that it is not degenerate. For example, $\varphi^*(\theta)\equiv 0$
and $\varphi^*_j(\theta) \equiv 0$ clearly satisfy BAR (\ref{eq:srbmbar}),
but $\varphi(\theta)\neq 0$. To this point, we show that
\red{\sout{if the sequence of probability measures corresponding to $\{\varphi^{(n_k)},\ k \geq 1\}$ was not tight.  If this were to be the case, then the fact that $\varphi^*(\theta)$ and its boundary counterparts satisfy BAR \eqref{eq:srbmbar} would not imply that $\varphi^*(\theta)=\varphi(\theta)$. }}
\begin{equation}
  \label{eq:mfgleftconti}
  \lim_{\theta\uparrow 0} \varphi^*(\theta)=1,
\end{equation}
which implies that the sequence of probability measures corresponding to $\{\varphi^{(n_k)},\ k \geq 1\}$ is tight; see Lemma~\ref{LEM:LAPLACETIGHTNESS}. It turns out that condition \eqref{eq:mfgleftconti} can \blue{be verified algebraically from} the fact that $\varphi^*(\theta)$ and $\varphi_j^*(\theta)$ $(j=1, \ldots, d)$ satisfy BAR \eqref{eq:srbmbar}. \blue{Most of the steps in this algebraic procedure are carried out in Proposition~\ref{lem:int_geq_boundary}.} Once we have this proposition, showing \eqref{eq:mfgleftconti} becomes straightforward; see for instance, the proof of (\ref{eq:mgfcon}). The proof of Proposition~\ref{lem:int_geq_boundary} requires that the reflection matrix of the SRBM be an ${\cal M}$-matrix. \blue{An ${\cal M}$-matrix is an invertible square matrix whose diagonal entries are non-negative, and off diagonal entries are non-positive \cite[Chapter 6]{BermPlem1979}. This condition is always satisfied in the GJN setting, because the reflection matrix of the SRBM has the form $(I - P^{\rs{t}})$, where $P$  routing matrix of the GJN.}

\blue{The procedure of considering a sequence $\varphi^{(n_k)}(\theta) \to \varphi^{*}(\theta)$ (and $\varphi^{(n_k)}_j(\theta) \to \varphi^*_j(\theta)$), and then verifying \eqref{eq:mfgleftconti} looks like an application of L\' evy's convergence theorem \cite[Section 18.1]{Will1991}, with one key difference. L\' evy's result says that if a sequence of characteristic functions (CF) converges, and the limit is continuous at zero, then the corresponding sequence of probability measures converges weakly to a limiting probability measure. Our use of MGFs in this paper, instead of CFs, is not accidental. Applying L\' evy's theorem requires knowing apriori that the sequence (or at least a subsequence) of CFs converges, and for a complicated model like a GJN this is nigh impossible to verify. In contrast, MGFs are monotone functions, and any sequence of MGFs always has a convergent subsequence due to Helly's selection principle \cite[Theorem 8.1]{Gut2005}.}


To prove Proposition~\ref{lem:prelimitMGFBAR} (that $\varphi^{(n_k)}(\theta)$ and its boundary counterparts satisfy BAR (\ref{eq:srbmbar}) asymptotically), we work with a continuous time Markov
process $X^{(n)}=\{X^{(n)}(t), t\ge 0\}$ that describes the dynamics of the $n$th GJN.  In
addition to the queue length process $L^{(n)}=\{L^{(n)}(t), t\ge 0\}$,
this Markov process also keeps track of the remaining interarrival times
and remaining service times at all stations.  Although $L^{(n)}$ is a jump
process taking values $\ell$ in $\Z_+^{d}$, the other component of
$X^{(n)}$ is a piecewise deterministic process taking values $(u,v)$
in $\R^{2d}_+$. Nevertheless, the stationary distribution of $X^{(n)}$
satisfies a BAR (\ref{eq:hard_BAR}) for all ``good'' test functions $f(\ell, u,
v)$. The BAR for $X^{(n)}$ has two components.  The first component involves the deterministic
process and the second involves the jump process; the latter is generally difficult to analyze. To handle this difficulty, we choose
$f(\ell, u, v)$ to be an exponential function of the state variable
$(\ell, u, v)$ of the form 
\begin{align}
  f(\ell, u, v) =e^{ \langle \theta, \ell\rangle + \langle \eta, u\rangle + \langle \zeta, v\rangle }, \label{eq:exptest}
\end{align}
where $(\theta, \eta, \zeta)\in \R^{3d}$ are parameters and $\br{\cdot, \cdot}$ is the Euclidean inner product (we actually use truncated versions of these functions to accommodate general interarrival and service time distributions, which may not have exponential moments, see \eqref{eq:testf}). By judiciously choosing $\eta=\eta(\theta)$ and $\zeta=\zeta(\theta)$ as functions of $\theta \in \R^d$, we eliminate the jump term
of the BAR (\ref{eq:hard_BAR}), leaving us with the jump free BAR \eqref{eq:SE 1}. To obtain Proposition~\ref{lem:prelimitMGFBAR} from this jumpless BAR \eqref{eq:SE 1}, we perform Taylor expansion on $\eta(\theta)$ and $\zeta(\theta)$ to obtain their quadratic approximations (Lemma~\ref{lem:concave 1}), and establish corresponding error bounds (Lemma~\ref{lem:uniform 1}).



Both \cite{BudhLee2009} and \cite{GamaZeev2006} focused on proving the
tightness of $\{r_nL^{(n)}(\infty), n\ge 1\}$ by ingenuously
constructing appropriate Lyapunov functions.  Both papers rely on the Lipschitz continuity of the Skorohod map corresponding to the SRBM for their tightness argument. Such a map does not exist in the
multiclass queueing network setting, which makes the generalization of results from \cite{BudhLee2009} and
\cite{GamaZeev2006} to the multiclass setting difficult. Gurvich
\cite{Gurv2014a} is the only paper that provides a sufficient
condition for proving (\ref{eq:steadyConver}) in the multiclass setting. In addition to the strong exponential moment assumption,
\cite{Gurv2014a} assumes a strong state space collapse (SSC) condition: (c) fluid solutions of a critically loaded fluid model converge
to their equilibria at a  ``uniform linear rate'' in finite time. \red{\sout{His
conditions are satisfied by several classes of service disciplines
including the queue-ratio discipline. However, his condition (c) precludes any results for the
FIFO service discipline.}}  Gurvich
\cite{Gurv2014a} focused on generalizing the approach in
\cite{GamaZeev2006}. In particular, he retained the strong exponential
moment assumption. In a recent paper, Ye and Yao~\cite{YeYao2016}
focused on generalizing the approach in \cite{BudhLee2009} to
resource-sharing networks that lie outside of the multiclass queueing
network setting. Using a multiclass queueing network example, in
Section 5 of \cite{YeYao2016}, the authors outline the steps needed
for generalizing their approach to the multiclass queueing network
setting.  The authors are able to keep the weak moment condition (b'),
relaxing condition (c) to condition (c'): fluid solutions converge to
their equilibria ``uniformly fast'', but not necessarily in finite
time. However, they imposed a strong ``bounded workload condition'', which
is difficult to check in general.

\red{\sout{In our proof, we also establish the tighness of $\{r_n
L^{(n)}(\infty), n\ge 1\}$, but our technique differs from the previous authors'. Instead of proving tightness directly, we focus on proving (\ref{eq:mfgleftconti}), which
implies the tigthness of $\{r_{n_k} L^{(n_k)}(\infty), k\ge 1\}$. We
prove (\ref{eq:mfgleftconti}) through algebraic relationships obtained
through BAR (\ref{eq:srbmbar}). This approach is much simpler than the
Lyapunov function approach in \cite{GamaZeev2006} and
\cite{BudhLee2009}, and hence, more amenable for generalization to
the multiclass setting.}}

\blue{The approach of \citet{GamaZeev2006} is known as the limit interchange argument, and has since been used by others to study steady-state approximations of various queueing systems.} In the single server setting, \citet{Kats2010} studied a multiclass single server queue with feedback, and \citet{ZhanZwar2008} studied a limited processor sharing queue. In the many-server setting, \citet{Tezc2008} considered a parallel-server system with multiple server pools and no customer abandonment, \citet{GamaStol2012} examined a multiclass, many-server queue with abandonment, where customer service and patience times are exponentially distributed with means varying between different customer classes, and \citet{DaiDiekGao2014} considered a many-server queue with abandonment, where service times follow a phase-type distribution. In recent years, several papers \cite{BravDai2017, BravDaiFeng2016, Gurv2014, GurvHuanMand2014, GurvHuan2016} have gone beyond limit theorems, and establish rates of convergence to the approximating distribution. The framework underlying those papers (except for \cite{GurvHuanMand2014}) is known as Stein's method \cite{Stei1986,ChenGoldShao2011, Ross2011}.

\blue{The limit interchange and Stein method frameworks represent the two general approaches used to establish convergence of steady-state distributions. Our paper adds a third approach to this set. Each approach has its own pros and cons. The Stein approach is able to provide rates of convergence, which is a step beyond just convergence, but this comes at a cost. Successfully applying it requires deeper knowledge about the underlying system than either the limit interchange approach, or the one presented in this paper. In particular, Stein's method has not been applied to queueing networks and has so far been limited to systems with a single station. The limit interchange approach has been the prominent method in the past decade.  At its core, it requires the use of a Lyapunov function to prove tightness. However, each stochastic system requires a separate Lyapunov function. Finding one is typically very difficult and can be considered an art. In contrast, the method in this paper is algorithmic in nature, and requires no guesswork to find any Lyapunov function. For instance, there is little wiggle room in choosing the exponential test function in \eqref{eq:exptest}, and our tightness argument in Section~\ref{sect:tightness} is algebraic and ``procedural''. In terms of generality of our method, multiclass queueing networks do add an extra layer of difficulty to our approach. Namely, the presence of SSC in those networks and the fact that the reflection matrix of the SRBM no longer has to be an $\cal M$-matrix. It is the subject of ongoing research to extend our approach to the multiclass setting.}

The rest of the paper is structured as follows. In \sectn{heavy}, we introduce the sequence of GJNs, the heavy traffic condition, describe the approximating SRBM, and state our main results. In \sectn{tractable}, we derive the BAR \eqref{eq:hard_BAR} for each GJN in the sequence, and introduce conditions on test functions under which the jump term \blue{there} disappears. Section~\ref{SECT:PRELIMIT} is devoted to proving Proposition~\ref{lem:prelimitMGFBAR}, which states that the MGFs of the queue lengths of the GJN approximately satisfy the BAR of the SRBM. In \sectn{tightness} we present Proposition~\ref{lem:int_geq_boundary}, which we use together with Proposition~\ref{lem:prelimitMGFBAR} to prove our main result, \thr{heavy traffic 1}, in \sectn{proofs}. We defer proofs of certain technical lemmas to the Appendix.

\dgreen{We advise the reader that the present document serves as an improved
version of the recently published paper \cite{BravDaiMiya2017}. The
improvement contains two minor changes. The first change is in the
first sentence of the proof of Theorem~\ref{thr:heavy traffic 1} in
Section~\ref{sect:proofs}. The existence of a subsequence $\{n''\}$
follows directly from Helly's selection principle; there is no need to
use the notion of ``weak compactness''. The second change is in the
proof of part \eqref{eq:iiclaimtightness} of
Lemma~\ref{LEM:LAPLACETIGHTNESS}, replacing reference
\cite{StadTrau1981} by \cite{Kall2001}.}

\subsection{Notation}
\label{sect:notation}
All random variables and stochastic processes are defined on a common
probability space $(\Omega, \mathcal{F}, \mathbb{P})$, and all stochastic processes $X = \{X(t), t \geq 0\}$ are assumed to be right continuous on $[0,\infty)$, and having left limits on $(0,\infty)$. For a sequence of random variables $\{Y_n,\ n \geq 1\}$ and a random variable $Y$, we write $Y_n \Rightarrow Y$ if $Y_n$ converge in distribution to $Y$. For an
integer $d \geq 1$, $\R^d$ denotes the $d$-dimensional Euclidean space,
and $\Z_+^d$ and $\R_+^d$ denote the spaces of $d$-dimensional vectors whose
elements are non-negative integers and non-negative real numbers, respectively. For vectors $x,y \in \R^d$, we write $x_i$ to denote the $i$th component of $x$, $1 \leq i \leq d$. Furthermore, we write $x \leq y$ if $x_i \leq y_i$ for all $i = 1, ... , d$ and we let $\br{x,y}$ be their Euclidean inner product. All vectors are understood to be column vectors. A function $f : \R^{d} \to \R$ is said to be non-decreasing if $x \leq y$ implies $f(x) \leq f(y)$. For a vector $x \in \R^d$, define the sup-norm $||x||_{\infty} = \sup_{i } \abs{x_i}$. For integers $a, b$ with $a > b$, we define $\sum_{i=a}^{b} = 0$. For integers $i,j$, we let $\delta_{ij}$ be the Kronecker delta; i.e.\ $\delta_{ij} = 1$ if $i=j$, and zero otherwise. We let $x^{\rs{T}}$ and $A^{\rs{T}}$ denote the transpose of a
vector $x$ and matrix $A$, respectively. We reserve $I$
for the identity matrix, $e$ for the vector of all ones and $e^{(i)}$
for the vector that has a one in the $i$th element and zeroes
elsewhere; the dimensions of these vectors will  be clear from the context.

\section{Heavy traffic approximation} \setnewcounter
\label{sect:heavy}
In this section, we introduce the generalized Jackson network and state 
the main results of this paper.

\subsection{Network description}
\label{sect:network}
To be able to state our main results, we first introduce
a generalized Jackson network, and define a Markov process that describes it. This network has $d$ stations, numbered $1,2,
\ldots, d$. Let $\J = \{1,2, \ldots, d\}$. Each station has a single
server that serves customers in the first-in-first-out (FIFO) manner. A station may have customers arriving from outside the network; we refer to such arrivals as external arrivals. A customer who
completes service at a station either goes to another station or leaves the network. The
following is a mathematical description of a GJN.

 Let $\sr{E}$ be the subset of $\J$ whose members are the stations that have external arrivals. External arrivals at station $i \in \sr{E}$ follow a renewal process with i.i.d.\ interarrival times 
\begin{align}
\{T_{e,i}(m),\ m=1, 2, \ldots, \} \label{eq:seq_arrival},
\end{align}
and we let $T_{e,i}$ be a nonnegative random variable having the distribution of the interarrival times. We assume that $T_{e,i}$ has finite variance and a non-zero mean. External arrivals at different stations are independent.

Service times at station $i$ are i.i.d.\ random variables 
\begin{align}
\{T_{s,i}(m),\ m=1,2, \ldots, \} \label{eq:seq_service},
\end{align}
and we let $T_{s,i}$ be a nonnegative random variable having the distribution of the service times. We assume that $T_{s,i}$ has finite variance and a non-zero mean. Service times at different stations are independent.
Service time sequences and external interarrival time sequences are assumed to be independent.

A customer that completes service at station $i \in \J$ goes to station $j\in \J$
with probability $p_{ij}$ or exits the network with probability
$p_{i0} = 1-\sum_{j\in \J}p_{ij}$, independently of everything else.  Let $P$ be the $d\times d$ square matrix whose $(i,j)$th entry is $p_{ij}$ for
$i,j \in \J$.

This queueing network is referred to as a generalized Jackson network
(GJN).  We now introduce a Markov process for describing the GJN. For
time $t \ge 0$, denote the number of customers, the residual external
arrival time, and the residual service time at station $i \in \J$ by
$L_{i}(t)$, $R_{e,i}(t)$ and $R_{s,i}(t)$, respectively. We set $R_{e,i}(t) = 0$ for $i \in \J \setminus \sr{E}$, and
$R_{s,i}(t) = T_{s,i}(m)$ if no customer is in service at
  station $i$ at time $t$, and the service time of the next customer at
  station $i$ is $T_{s,i}(m)$.  In Section 2.1 of \cite{Dai1995} and
Section 2.1.1 of \cite{BudhLee2009}, $R_{s, i}(t)$ is defined to be
zero if no customer is in service at time $t$. Our definition will make it slightly easier to derive condition (\ref{eq:cond2}) in \sectn{dynamics} to annihilate the jump term in \eqref{eq:hard_BAR} there.

Denote the vectors whose entries are $L_{i}(t)$,
$R_{e,i}(t)$ and $R_{s,i}(t)$  by $L(t)$, $R_{e}(t)$ and
$R_{s}(t)$,  respectively. Throughout the paper we refer to $\{L(t), t \geq 0\}$ as the queue length process (even though it includes customers currently in service as well). Let $X(t) = ({L}(t), {R}_{e}(t),
  {R}_{s}(t))$, then $X =\{X(t), t\ge 0\}$ is a Markov process
with state space $\dd{Z}_{+}^{d} \times \dd{R}_{+}^{2d}$.
   
The GJN is a natural generalization of the Jackson network, but its
stationary distribution is hard to get. This motivates the study of heavy
traffic approximations in \cite{Reim1984} and \cite{John1983}. To
introduce the notation of heavy traffic, we introduce a sequence of \blue{generalized}
Jackson networks in the next section.



\subsection{A Sequence of Networks and Their Assumptions}
\label{sect:sequence}

Consider a sequence of GJNs indexed by $n =1,2,\ldots$. We denote the
$n$th GJN by superscript $^{(n)}$. For example, $X(t), {L}(t),
{R}_{e}(t), {R}_{s}(t)$ are denoted by $X^{(n)}(t)$, ${L}^{(n)}(t)$,
${R}_{e}^{(n)}(t)$, ${R}_{s}^{(n)}(t)$, respectively. We assume the
routing matrix $P=(p_{ij})$ is independent of $n$.  The networks are
assumed to be open, i.e., the matrix $(I-P)$ is invertible; the
inverse is given by
\begin{equation}
  \label{eq:invertible}
(I-P)^{-1}=  I + P +P^2 + \ldots.
\end{equation}
For $i \in \sr{E}$, we denote the mean and variance of $T_{e,i}^{(n)}$ by $1/\lambda^{(n)}_{e,i}$ and $(\sigma_{e,i}^{(n)})^2$, respectively.  For notational simplicity, we adopt the conventions that $\lambda_{e, i} = 0$ and $(\sigma_{e,i}^{(n)})^2 = 0$ for $i \in \J \setminus \sr{E}$. Similarly, for $j \in \J$, we denote the mean and variance of $T_{s,j}^{(n)}$ by $1/\lambda^{(n)}_{s,j}$ and $(\sigma_{s,j}^{(n)})^2$, respectively. Let $\lambda^{(n)}_{a,i}$ for $i \in
\J$ be the solution of the traffic equation:
\begin{align*}
  \lambda^{(n)}_{a,i} = \lambda^{(n)}_{e,i} + \sum_{j \in \J} \lambda^{(n)}_{a,j} p_{ji}, \qquad i \in \J.
\end{align*}
Then $\lambda^{(n)}_{a,i}$ can be interpreted as the total arrival rate at station $i$.  The traffic
equation  can be written as the vector valued equation:
\begin{equation}
  \label{eq:traffic 2}
\lambda_{a}^{(n)} = \lambda_{e}^{(n)} + P^{\rs{t}}
\lambda_{a}^{(n)},
\end{equation}
where  $P^{\rs{t}}$ is the transpose of $P$. Equation (\ref{eq:traffic 2}) has a unique  solution  given by
$\lambda_{a}^{(n)}=(I-P^{\rs{t}})^{-1}\lambda_{e}^{(n)}$.

 We assume the following heavy traffic conditions: there exists a positive vector $b \in \R^d_+$, and a sequence of positive numbers $r_{n}$ such that
\begin{align}
\label{eq:heavy 1}
 & \lambda^{(n)}_{s} - \lambda^{(n)}_{a} = b\, r_{n}, \qquad n \ge 1,\\
\label{eq:heavy 2}
 & \lim_{n \to \infty} {r}_{n} = 0.
\end{align}
 It will be
convenient to express condition \eq{heavy 1} in terms of the
primitive data $\lambda_{e}^{(n)}$ and $\lambda_{s}^{(n)}$.  For this, we substitute $\lambda^{(n)}_{a}=\lambda^{(n)}_{s}-b\, r_n$ into both
sides of \eq{traffic 2} to get
\begin{align}
\label{eq:heavy 3}
  \lambda_{s}^{(n)} - \lambda_{e}^{(n)} - P^{\rs{T}} \lambda_{s}^{(n)}= r_{n} Rb,
\end{align}
where 
\begin{equation}
  \label{eq:R}
  R = I-P^{\rs{T}}.
\end{equation}
Note that $r_{n}$ is chosen to be $1/\sqrt{n}$ in \cite{Reim1984} and
much of the literature as well, but it is intuitively clear that this
is not essential as long as \eqref{eq:heavy 1} holds and $r_{n}$ converges to zero as $n \to \infty$. For
example, some authors take $r_{n} =1/n$ (see, e.g.,
\cite{ChenMand1991}). In this paper, we do not make any specific
choice for $r_{n}$; this conveys
the same spirit of \citet{King1962}'s heavy traffic approximation (see
\cite{Miya2015} for details). 
 We make the following moment assumptions on the sequence of networks:
\begin{eqnarray}
  \label{eq:momentarr}
&&    \sigma^{(n)}_{e, i} \to     \sigma_{e, i}<\infty, \quad \lambda_{e, i}^{(n)}\to \lambda_{e, i} > 0 \quad \text{ for each } i \in {\cal E}, \\
&&    \sigma^{(n)}_{s, j} \to     \sigma_{s, j}<\infty  \quad \text{ for each } j \in {\J}.\label{eq:momentser}
\end{eqnarray}
In addition, we assume
\begin{eqnarray}
&&  \big\{ \big(T^{(n)}_{e,i}\big)^2, n \ge  1\big\} \text{ is uniformly integrable  for each }  \quad i\in {\cal E}, \label{eq:uniarr}\\
&&  \big\{ \big(T^{(n)}_{s,j}\big)^2, n \ge  1\big\} \text{ is uniformly integrable  for each }  \quad j\in {\cal J}. \label{eq:uniser}
\end{eqnarray}
Following traffic equation \eq{traffic 2}, conditions  \eq{heavy 1},  \eq{heavy 2}, and \eq{momentarr} imply  that
  \begin{equation}
    \label{eq:lamdaconverge}
\lambda^{(n)}_{a,j}\to \lambda_{a,j}, \quad \text{ and } \quad \lambda^{(n)}_{s,j}\to \lambda_{s,j}=\lambda_{a,j} \quad \text{ for } j\in \J,
  \end{equation}
  where $\lambda_a=(I-P^{\rs{t}})^{-1}\lambda_{e}$.

The diffusion approximation focuses on the sequence $L^{(n)} = \{L^{(n)}(t), t\ge 0\}$, which is the first component of $X^{(n)}$. Clearly, $L^{(n)}$ is not a Markov process in general, but the standard approach in the literature (e.g. \cite{John1983, Reim1984}) shows that the diffusion-scaled process $Z^{(n)}=\{Z^{(n)}(t), t\ge 0\}$, defined as
\begin{align}
\label{eq:diffusion scaling}
  Z^{(n)}(t) = r_{n} {L}^{(n)}\big(t/{r_{n}^{2}} \big), \quad t\ge 0,
\end{align}
converges in distribution to a semimartingale reflecting Brownian motion (SRBM) $Z=\{Z(t), t\ge 0\}$ (to be defined in Section~\ref{sect:main}).  The heavy traffic assumptions \eq{heavy 1} and \eq{heavy 2} are crucial for the time and space scalings in \eqref{eq:diffusion scaling} to be correct.

 We also assume that for each $n$,
\begin{equation}
  \label{eq:positiveRecurrent}
  X^{(n)} = \{X^{(n)}(t), t\ge 0\} \text{ is positive Harris recurrent}.
\end{equation}
This assumption is satisfied under heavy traffic condition \eq{heavy
  1} and some additional regularity assumptions on the interarrival time
distributions; see, for example,  Theorem 3.8 of \citet{DownMeyn1994} and Theorem~5.1 of \citet{Dai1995}. Since $X^{(n)}$ is positive Harris
recurrent, it has a unique stationary distribution. We let $X^{(n)}(\infty)$ be the vector having that stationary distribution.

Since $X^{(n)}$ is assumed to have a stationary distribution,
$L^{(n)}$ has a stationary distribution. 
We use $L^{(n)}(\infty)$ and $Z^{(n)}(\infty)$ to
denote the random vectors having the stationary distributions of
$L^{(n)}$ and $Z^{(n)}$, respectively. Note that the stationary distribution of $Z^{(n)}$ is independent of the time scaling in \eqref{eq:diffusion scaling} for each fixed $n$. Furthermore, it is clear that $Z^{(n)}(\infty) \stackrel{d}{=} {r}_{n} L^{(n)}(\infty)$. For future reference, 
we use $\pi^{(n)}$  to denote the stationary distribution of $Z^{(n)}$. As stated in \eqref{eq:steadyConver}, the primary result of this paper is to prove that $Z^{(n)}(\infty) \Rightarrow Z(\infty)$, where $Z(\infty)$ has the stationary distribution of the SRBM $Z$, which we now define.


\subsection{Semimartingale Reflecting Brownian Motions and BAR}
\label{sect:main}
Recall the matrix $R$, defined in (\ref{eq:R}), and set
\begin{align} \label{eq:mu}
\mu=-Rb \in \R^d,
\end{align} 
where $b$ is given in
\eq{heavy 1}. Let $\Sigma= (\Sigma_{ij}) $ be a $d\times d$ symmetric matrix
given by
\begin{align}\label{eq:Sig}
  \Sigma_{ij} & = \sum_{k \in \J} \lambda_{s,k} \big[ p_{ki} (\delta_{ij}
  - p_{kj}) \\
  & \quad + \lambda_{s, k}^{2} \sigma_{s, k}^{2} (p_{ki} -
  \delta_{ki}) (p_{kj} - \delta_{kj}) \big] + \lambda_{e,i}^{3} \sigma_{e, i}^{2}
  \delta_{ij}, \nonumber
\end{align}
where $\delta_{ij}$ is the Kronecker delta, $p_{ij}$ are
the routing probabilities in the GJN, and the quantities $\lambda_{e,i}$,
$\sigma_{e,i}$, $\sigma_{s,i}$, and $\lambda_{s,i}$ are given in
\eq{momentarr}, \eq{momentser}, and
\eq{lamdaconverge}, respectively. The matrix $\Sigma$ is
always non-negative definite. Throughout the document, we assume that
\begin{equation}
  \label{eq:Sigmapd}
  \Sigma \text{ is positive definite,}
\end{equation}
\blue{which is a standard assumption in both \cite{GamaZeev2006, BudhLee2009}.}
Associated with the data $(\mu, \Sigma, R)$ is $Z=\{Z(t),t\ge 0\}$, a semimartingale
reflecting Brownian motion (SRBM) that satisfies
\begin{align}
&  {Z}(t) = {\xi}(t)+ R {Y}(t), \quad t\ge 0, \label{eq:srbm1}\\
&  {\xi} \text{ is a  $d$-dimensional Brownian motion with drift $\mu$} \\
&\text{ and covariance matrix $\Sigma$,} \notag \\
& Y(0)=0, \text{ each component of $Y$ is non-decreasing, and }  \\
&  \int_0^\infty Z_j(t) dY_j(t) = 0, \quad j\in \J. \label{eq:srbm4}
\end{align}
The matrix $R$ is called the reflection matrix of $Z$. Recall that an ${\cal M}$-matrix is an invertible square matrix whose diagonal entries are non-negative, and off diagonal entries are non-positive \cite[Chapter 6]{BermPlem1979}. Since $R$ in \eq{R} is an ${\cal M}$-matrix, it follows from
\cite{HarrReim1981} that the SRBM $Z$ exists and is unique as a
strong solution to \eq{srbm1}--\eq{srbm4}. 

Again, because $R$ is an ${\cal M}$-matrix and condition 
\begin{equation}
  \label{eq:driftcondidtion}
 R^{-1}\mu = -b <0
\end{equation}
is satisfied, 
it follows from \cite{HarrWill1987} that the SRBM $Z$ has a unique stationary distribution $\pi$ on $(\R^d_+, {\cal B}(\R^d_+))$, where $\sr{B}(\dd{R}_{+}^{d})$ is the Borel field of $\dd{R}_{+}^{d}$. We now discuss the characterization of $\pi$.

Let $\dd{E}_{\pi}$ denote the expectation when ${Z}(0)$ has distribution $\pi$. For $j \in \J$, define the boundary probability measure $\pi_{j}$ by
\begin{align*}
  \pi_{j}(B) =\frac{1}{\E_{\pi} [Y_j(1)]} \dd{E}_{\pi} \left[ \int_{0}^{1} 1(Z(t) \in B) dY_{j}(t)\right], \qquad B \in \sr{B}(\dd{R}_{+}^{d}).
\end{align*}
We know that $\E_{\pi} [Y_j(1)] < \infty$ by \cite[Theorem 1]{HarrWill1987}. Note that $Y_{j}(t)$ increases only on the face 
 \begin{align*}
   F_{j} = \{{x} \in \dd{R}_{+}^{d}: x_{j} = 0\}, \qquad j \in \J.
 \end{align*}
 Therefore, $\pi_{j}$ concentrates on $F_{j}$, but we define it
 on $\dd{R}_{+}^{d}$ for notational simplicity.

 The following lemma gives a characterization of the stationary distribution
 $\pi$ and its associated boundary measures $\pi_1$, $\ldots$, $\pi_d$.
 \begin{lemma}
\label{lem:barneccsuff}
 (a) The stationary distribution  $\pi$ and its associated boundary measures $\pi_1$, $\ldots$, $\pi_d$ must 
satisfy  the following basic adjoint relationship (BAR)
\begin{align}
  \label{eq:bar}
  & \int_{\R^d_+} G f(x) \pi(dx) + \sum_{j\in \J} \E_{\pi} [Y_j(1)] \int_{\R^d_+}\langle  \nabla f(x), R^{(j)}\rangle  \pi_j(dx)=0, \\
& \qquad \text{ for each }f\in C^2_b(\R^d_+), \notag 
\end{align}
where $R^{(j)}$ is the $j$th column of $R$, $C^2_b(\R^d_+)$ is the
set of functions $f$ on $\R^d_+$ such that $f$, its first order
derivatives, and its second order derivatives are bounded and
continuous, and 
\begin{equation}
  \label{eq:G}
  Gf (x) = \frac{1}{2}\sum_{i,j=1}^d \Sigma_{ij} \frac{\partial ^2f}{\partial x_i \partial x_j}(x) + \sum_{j=1}^d \mu_j \frac{\partial f}{\partial x_j} (x).
\end{equation}
(b) Conversely, assume that $\pi$, $\pi_1$, $\ldots$, $\pi_d$
are probability measures on $\R^d_+$ satisfying BAR
\eq{bar}. Then $\pi$ must be the stationary distribution of
the $(\mu, \Sigma, R)$-SRBM, and $\pi_1$, $\ldots$, $\pi_d$ the corresponding
boundary measures.
 \end{lemma}
 \begin{proof}
   Part (a) follows from \cite{HarrWill1987}. Part (b) follows from
   \cite{DaiKurt1994}.
 \end{proof}

 We denote the moment generating functions (MGFs) of $\pi$ and $\pi_{j}$ by $\varphi$ and $\varphi_{j}$, respectively. Namely, for $\theta \in \R^d$ with $\theta \le 0$,
\begin{align} \label{eq:mgfsrbmdef}
 & \varphi({\theta}) = \dd{E}_{\pi}[ e^{\langle {\theta}, {Z}(0)\rangle}], \\
 & \varphi_{j}({\theta}) = \frac{1}{\E_{\pi} [Y_j(1)]} \dd{E}_{\pi} \left[ \int_0^1  e^{\langle
      {\theta}, {Z}(t)\rangle}dY_{j}(t)\right], \quad j \in \J. \nonumber
\end{align}
We also define
\begin{align}
\label{eq:gamma 1}
 & \gamma({\theta})=\frac{1}{2}\langle{\theta}, \Sigma{\theta}\rangle + \br{-Rb, {\theta}}, \\
 & \gamma_{j}({\theta}) = \br{R^{(j)}, {\theta}}, \qquad  {\theta} \in \dd{R}^{d},\ j \in \J. \notag
\end{align}
Plugging $f(x)=e^{\br{\theta, x}}$ into \eqref{eq:bar}, the BAR \eqref{eq:bar} becomes
\begin{align} \label{eq:intermed_srbm_mgf_bar}
  &\gamma({\theta}) \varphi({\theta}) + \sum_{j=1}^d \E_{\pi} [Y_j(1)]  \gamma_{j}({\theta}) \varphi_{j}\bigl({\theta}\bigr) = 0, \\ 
 & \qquad \text{ for each } \theta \in \R^d \text{ with } \theta\le 0, \notag 
\end{align}
and we now use it to show that
 \begin{equation}
  \label{eq:bEy}
 \E_{\pi} [Y_j(1)] = b_j, \quad j \in \J.
\end{equation}
To do so, we let $\theta = \alpha e^{(k)}$, where $\alpha \leq 0$ is a real number and $k \in \J$. Dividing both sides of \eqref{eq:intermed_srbm_mgf_bar} by $\alpha$ and taking $\alpha \uparrow 0$, we obtain
\begin{align*}
  -\br{Rb, {e^{(k)}}} + \sum_{j=1}^d \E_{\pi} [Y_j(1)]\br{R^{(j)}, e^{(k)}} = \sum_{j=1}^d \big(\E_{\pi} [Y_j(1)] - b_j \big) r_{kj} = 0
\end{align*}
for all $k \in \J$, where $r_{kj}$ is the $(k,j)$th entry of $R$. In vector form, this is equivalent to 
\begin{align*}
R\big( \E_{\pi} [Y(1)] - b \big) = 0,
\end{align*}
and since $R$ is invertible, \eqref{eq:bEy} follows. Therefore, \eqref{eq:intermed_srbm_mgf_bar} can be rewritten in the more practical form 
\begin{align}
\label{eq:srbmbar}
  \gamma({\theta}) \varphi({\theta}) + \sum_{j=1}^d b_j\gamma_{j}({\theta}) \varphi_{j}\bigl({\theta}\bigr) = 0, \quad \text{ for each } \theta \in \R^d \text{ with } \theta\le 0.
\end{align}
We refer to this as the MGF version of the BAR.

It is shown in the appendix of the arXiv version of \cite{DaiMiyaWu2014}
that the MGF version of the BAR (\ref{eq:srbmbar}) and the standard version
the BAR \eq{bar} are equivalent.
Thus, it follows from the characterization obtained in \cite{DaiKurt1994} that $\varphi({\theta})$ and $\varphi_{j}\bigl({\theta}\bigr)$ are the unique moment generating functions that satisfy \eqref{eq:srbmbar}.

\subsection{Main Results}
\label{sect:results}
Recall that $\pi^{(n)}$ is the stationary distribution of $Z^{(n)}$ and for $j \in \J$, define $\pi_j^{(n)}$ to be the distribution of $\big[Z^{(n)}(\infty) \big| Z^{(n)}_j(\infty) = 0\big]$. We are now ready to present our main results. 

\begin{theorem}
\label{thr:heavy traffic 1}
Consider the sequence of GJNs indexed by $n$. Assume the heavy traffic conditions
\eq{heavy 1} and  \eq{heavy 2}, positive recurrence condition
  \eq{positiveRecurrent}, and moment conditions
  \eq{momentarr}--\eq{uniser}  hold. Then, 
\begin{enumerate}[(a)]
\item As $n \to
\infty$, the stationary distribution ${\pi}^{(n)}$ of $Z^{(n)}$ converges weakly to $\pi$, the stationary
distribution of the SRBM $Z$.
\item For each $j\in \J$,
${\pi}^{(n)}_{j}$ converges weakly to $\pi_{j}$, the
corresponding boundary measure on $F_j$.
\end{enumerate} 
\end{theorem}

Part (b) appears to be new. As mentioned in \sectn{sequence}, Part (a) of the theorem is known, but we prove it using a new approach. To elaborate on this approach, we first discuss the existing approach that uses process limit and tightness.

Assuming the distribution of ${Z}^{(n)}(0)$ converges,
\citet{Reim1984} proves that the scaled queue length process
$\{{Z}^{(n)}(t), t\ge 0\} $ converges in distribution to the
SRBM $Z$; we refer to this as the process limit. Namely,
\begin{align}
\label{eq:process limit 1}
  {Z}^{(n)}(\cdot) \Rightarrow  {Z}(\cdot) \quad \text{ as $n \to \infty$,}
\end{align}
Also, it is proved in \cite{HarrWill1987} that 
\begin{equation}
  \label{eq:Zconvsd}
  Z(t) \Rightarrow Z(\infty) \quad \text{ as } t\to\infty,
\end{equation}
where $Z(\infty)$ is a random vector having distribution $\pi$.
From \eq{process limit 1} and \eq{Zconvsd}, it is not surprising that ${Z}^{(n)}(\infty) \Rightarrow Z(\infty)$, which is the content of part (a) of Theorem~\ref{thr:heavy traffic 1}.
Here, $ {Z}^{(n)}(\infty)$ is the random vector having the
stationary distribution of $ \{{Z}^{(n)}(t), t\ge
0\}$. As we mentioned earlier, part (a) of Theorem~\ref{thr:heavy traffic 1} was already proved by \citet{GamaZeev2006}, and \citet{BudhLee2009}. Both sets of authors rely on the process limit result \eq{process limit 1}, and reduce the problem to showing tightness
of the stationary distributions $\{\pi^{(n)}:n=1, 2,\ldots\}$. Together with the process limit, this tightness implies \eqref{eq:steadyConver}. Our proof of Theorem
\ref{thr:heavy traffic 1} will not use the process limit \eq{process limit 1}. Instead, we work directly with the BAR associated with each network in the sequence, which we now derive in \sectn{tractable}.
  


\section{Network dynamics and basic adjoint relationship} \setnewcounter
\label{sect:tractable}
In \sectn{dynamics} we derive the BAR for the GJN and describe a special class of test functions for which the BAR reduces to a very simple expression. In \sectn{test}, we focus on a family of exponential test functions that belong to this special class.
This lays down the foundation for Section~\ref{SECT:PRELIMIT}, where we compare the BAR of the GJN to that of the approximating SRBM.

\subsection{Network Dynamics and Tractable BAR}
\label{sect:dynamics}
In this section, we derive the BAR (\ref{eq:hard_BAR}) for a GJN using sample path dynamics of the network. The main result of this section is Lemma~\ref{lem:MCbar} below. It describes a special class of functions for which the BAR has a simple form. This lemma can be compared to the BAR in \citet{Miya2017}, which studies a heterogeneous multiserver
queue. However, our approach is different and we make detailed comments on this difference at the end of
this section.

 We first describe the network dynamics in terms
of network primitives.  We fix a network in the sequence and temporarily drop the index $n$ for the remainder of the section. For station $j \in \J$, we introduce the sequence of
i.i.d.\ random vectors 
\begin{align}
\{\phi^{(j)}(m) \in \Z^{d}_+,\ m =1, 2, \ldots\} \label{eq:routing_decisions}
\end{align} 
with
\begin{align*}
&\dd{P}(\phi^{(j)}(1)=e^{(k)})=p_{jk}, \quad k \in {\cal J} \\ 
&\dd{P}(\phi^{(j)}(1)=0)=1-\sum_{k\in{\cal J}}p_{jk} = p_{j0},
\end{align*}
where $p_{jk}$ is the probability that a customer goes to station $k$ after completing service at station $j$. These random vectors represent the routing of customers after completion of service at station $j$.

Recall the definition of $X = \{X(t), t \geq 0\}$ from Section~\ref{sect:network}, and let 
\begin{align*}
X(0) = (L(0), R_e(0), R_s(0))\in \Z_+^d \times \R_+^{2d}
\end{align*} 
be the initial state of the network. We recall the sequences of interarrival and service times from \eqref{eq:seq_arrival} and \eqref{eq:seq_service}, and for each integer $q \geq 1$ we define the primitive processes
\begin{align}
 & U_i(q)= R_{e,i}(0) + \sum_{m=1}^{q-1}
  T_{e,i}(m), \quad i\in {\cal E}, \label{eq:arrival_times} \\
& V_j(q)= R_{s,j}(0) +  \sum_{m=1}^{q-1} T_{s,j}(m), \quad j\in {\cal J}, \label{eq:service_cumul_times} \\
 & \Phi^{(j)}(q) = \sum_{m=1}^q \phi^{(j)}(m), \quad j\in {\cal J}. \label{eq:routing_vector}
\end{align}
For station $j \in \J$, and time $t \geq 0$, let $E_j(t)$ and $D_j(t)$ be the total number of external arrivals and total number of departures, respectively, on the interval $[0,t]$. Let 
\begin{align}
B_j(t) = \int_0^t 1(L_j(s)\ge 1) ds, \quad j \in \J \label{eq:idle_time}
\end{align}
be the cumulative busy time of the server at station $j$. We can then define 
\begin{align}
S_j(q) = \inf \{t > 0 :\ B_j(t) = V_j(q)\}, \quad j \in \J, q \geq 1, \label{eq:service_times}
\end{align} 
to be the time of the $q$th service completion at station $j$. Recalling that only stations in $\sr{E}$ have external arrivals, we see that for every $t \geq 0$ and on every sample path, 
\begin{align}
 & E_i(t) = \max\big\{q\in \Z_+: U_i(q)\le t\big\}, \quad i\in {\cal E},   \label{eq:arrival_process}\\
 & E_i(t) \equiv 0, \quad i\in {\cal J} \setminus {\cal E}, \notag \\
& D_j(t) = \max\big\{q\in \Z_+: S_j(q) \leq t\big\}, \quad j \in \J. \label{eq:departure_process}
\end{align}
Furthermore, one may check that the queue lengths, residual interarrival times and residual service times satisfy
\begin{align}
& L_j(t) = L_j(0) + E_j(t) -D_j(t) + \sum_{k\in {\cal J}} \Phi^{(k)}_j(D_k(t)) , \quad j \in \J, \label{eq:dynamics_queue}\\
& R_{e,i}(t) = U_i({E_i(t)+1})-t, \quad i\in {\cal E}, \notag \\
& R_{e,i}(t) \equiv 0, \quad i\in {\cal J} \setminus {\cal E},\notag  \\
& R_{s,j}(t) = V_j({D_j(t)+1})-B_j(t), \quad j \in \J. \notag 
\end{align}
In the last equation, we have adopted the convention that when station $j$ is empty at time $t$, the remaining service time $R_{s,j}(t)$ is set to be the service time of the next
customer at this station. Clearly, 
\begin{align} \label{eq:determenistic_dynamics}
 &  \frac{\partial }{\partial t}R_{e,i}(t) = -1, \text{ for } i \in {\cal E} \\
 &  \frac{\partial }{\partial t}R_{s,j}(t) = -1(L_j(t)>0), \quad \text{ for } j\in {\cal J}, \notag
\end{align}
where the derivatives at a jump time $t$ are interpreted as the
right derivatives. With the network dynamics rigorously defined, we proceed to derive a BAR for the network. 

\blue{Let $\mathcal{D}$ be the space of all bounded functions $f(\ell, y):\Z_+^d \times \R_+^{2d} \to \R$ defined as follows. For any $i = 1, \ldots , 2d$, fix $(x,(y_k)_{k\neq i}) \in \Z_+^d \times \R_+^{2d-1}$ and view $f(x,y)$ as a one dimensional function in the $y_i$ component. We require that this one dimensional function be continuously differentiable at all but finitely many points, and have bounded derivatives whose bound is independent of the point $(x, (y_k)_{k \neq i})$. For instance, $\cal D$ contains the space of all bounded functions $f:\Z_+^d \times \R_+^{2d} \to \R$, such that for any $\ell \in \Z^d_+$, the function $f(\ell,\cdot): \R^{2d}_+ \to \R$ is  continuously differentiable with a bounded derivative, whose bound is independent of $\ell$. The reason for enforcing the component-wise conditions on the set $\cal D$ is that the test functions we use in this paper always have some form of truncation, which prevents them from being everywhere differentiable in the variable belonging to $\R_+^{2d}$.}
%

Now for any $f\in \mathcal{D}$, and any interval $[t,t+h] \subset \R_+$ free of jumps, we can use the fundamental theorem of calculus together with \eqref{eq:determenistic_dynamics} to see that
\begin{displaymath}
  f(X(t+h))-f(X(t)) = \int_t^{t+h} {\cal A}f(X(s))\,ds,
\end{displaymath}
where 
\begin{align}
  \label{eq:A}
  {\cal A}f(x) = -\sum_{i\in {\cal E}}\frac{ \partial f}{\partial u_i}(x) 
 - & \sum_{j\in {\cal J}}\frac{ \partial f}{\partial v_j}(x)1(\ell_j>0), \\
 & x=(\ell, u, v) \in \Z_+^d \times \R_+^{d} \times \R_+^{d}. \notag
\end{align}
For the remainder of this section, we let $\nu$ denote the stationary distribution of $X$, and let $\Prob_{\nu}$ be the probability measure conditioned on $X(0)$ having the distribution $\nu$. To deal with the jumps of $X$, we introduce some notation.

We know that the jumps of the process $X$ correspond to external arrivals and departures at various stations. We use the term event to refer to a single external arrival or departure. At time instance $s$, there may be simultaneous events. A jump of $X$ at time $s$ constitutes all the events that occur at $s$. Since we assumed that $\E T_{e,i} > 0$ and $\E T_{s,j} > 0$ for all $i \in \sr{E}$ and $j \in \J$, it follows from basic renewal theory (see for instance \cite[Theorem 3.3.1]{Resn1992}) that
\begin{align}
\sum_{i \in \sr{E}} \E_{\nu} \big[E_i(t) - E_i(0)\big] + \sum_{j \in \J} \E_{\nu} \big[D_j(t) - D_j(0)\big] < \infty. \label{eq:finite_intensities}
\end{align}
The finiteness of $\E_{\nu} \big[D_j(t) - D_j(0)\big]$ comes from the fact that $\{D_j(t),\ t \geq 0\}$ is dominated by the renewal process corresponding to the times $\{V_j(q),\ q \geq 1\}$.
Therefore, for every $t > 0$ we know that $\Prob_{\nu}$-almost surely, the process $X$ has finitely many events on the interval $(0,t]$. It follows that $\Prob_{\nu}$-almost surely,
\begin{align}
&\text{the process $X$ has countably many jumps on the interval $(0,\infty)$}, \notag \\
&\text{and every jump instant on this interval has finitely many events.} \label{eq:finite_events_assumption}
\end{align}
In what follows, we deal only with the sample paths where \eqref{eq:finite_events_assumption} holds. Suppose now that $k$ events occur simultaneously at time $s$. We can order them in an arbitrary manner, provided that we do not violate the network dynamics. For example, if station $s$ is empty at time $s-$, and experiences both an arrival and departure at time $s$, then the arrival must happen first.  The particular order assigned to the simultaneous events does not matter, because \eqref{eq:finite_events_assumption} implies there are always finitely many events at a time instance. We can therefore order all of the events that occur on the interval $(0,\infty)$, and represent them as $\delta_1 < \delta_2 < \ldots$, where $\delta_m$ represents the $m$th event to occur after time $0$. We let $T(\delta_m)$ represent the time at which the $m$th event happens.

Now for integer $m \geq 1$, we write $X_{\delta_m}$ to represent the value of the process immediately after the $m$th event has been applied to it. Our use of $X_{\delta_m}$ as opposed to $X(T(\delta_m))$ is intentional. If $\delta_{m_1}$ and $\delta_{m_2}$ represent two simultaneous arrivals to some station, then $X(T(\delta_{m_1})) =  X(T(\delta_{m_2}))$, but $X_{\delta_{m_1}} \neq X_{\delta_{m_2}}$. We also write $X_{\delta_m-}$ to represent the value of the process $X$ right before the $m$th event is applied to it, but after the first $m-1$ events have been applied. 

From \eqref{eq:finite_intensities}, we see that
\begin{align} \label{eq:finite_total_jumps}
& \E_{\nu}\bigg[\sum_{m=1}^{\infty} 1(0 < T(\delta_m) \leq t) \bigg] \\
&\quad = \sum_{i \in \sr{E}} \E_{\nu} \big[E_i(t) - E_i(0)\big] + \sum_{j \in \J} \E_{\nu} \big[D_j(t) - D_j(0)\big] < \infty. \notag
\end{align}
By isolating times when jumps occur, one can verify that for any $t > 0$ and $f \in \mathcal{D}$, 
\begin{align*}
& f(X(t)) = f(X(0)) + \int_0^t {\cal A}f(X(s))ds \\
& \qquad +\sum_{m=1}^{\infty} \big(f(X_{\delta_m})- f(X_{\delta_m-}) \big) 1(0 < T(\delta_m)  \leq t) , \quad  \Prob_{\nu}\text{-almost surely}.  \notag
\end{align*}
Since $f \in \mathcal{D}$ is bounded, we can take the expectation under $\nu$ to see that 
\begin{align*}
\E_{\nu} & \bigg[\int_0^t {\cal A}f(X(s))ds \bigg] \\
& + \E_{\nu}\bigg[\sum_{m=1}^{\infty} \big(f(X_{\delta_m})- f(X_{\delta_m-}) \big) 1(0 < T(\delta_m)  \leq t) \bigg]
= 0.
\end{align*}
Furthermore, we know that ${\cal A}f(x)$ is bounded for all $x \in \Z_+^d \times \R_+^{2d}$, meaning we can apply the Fubini-Tonelli theorem to interchange the expectation with the integral in the first term. Stationarity implies that 
\begin{align*}
\E_\nu\big[{\cal A}f(X(s))\big] = \E_\nu\big[{\cal A}f(X(0))\big], \quad \text{ for all $s \in [0,t]$}.
\end{align*}
 We therefore have the intermediate result
\begin{align*} 
\E_{\nu} \big[ {\cal A}f(X(0))\big]+   \E_{\nu}\bigg[ \frac{1}{t} \sum_{m=1}^{\infty} \big(f(X_{\delta_m})- f(X_{\delta_m-}) \big) 1(0 < T(\delta_m)  \leq t) \bigg] = 0.
\end{align*}
We now use \eqref{eq:finite_total_jumps} and the boundedness of $f \in \mathcal{D}$ to apply the Fubini-Tonelli theorem again and arrive at the BAR for the GJN:
\begin{align} \label{eq:hard_BAR}
 & \E_{\nu} \big[ {\cal A}f(X(0))\big] \\
 & +  \frac{1}{t}\sum_{m=1}^{\infty} \E_{\nu}\Big[\big(f(X_{\delta_m})- f(X_{\delta_m-}) \big) 1(0 < T(\delta_m)  \leq t) \Big] = 0, \quad f \in \mathcal{D}. \notag 
\end{align}
\blue{Before proceeding further, let us pause and discuss the implications of \eqref{eq:hard_BAR}. The equation above encodes information about the distribution $\nu$. To access this information, we choose various test functions $f \in \mathcal{D}$ and plug them into \eqref{eq:hard_BAR}. However, \eqref{eq:hard_BAR} is of limited practical use for most test functions because the jump term is hard to handle analytically. To get around this difficulty, we now describe how to engineer $f \in \mathcal{D}$ so that the jump terms above vanish. Roughly speaking, the idea is to ensure that\begin{displaymath}
  \E [f(X(s))| X(s-)=x] = f(x)
\end{displaymath}
at each jump time $s$ and for each state $x$, i.e.\ the test function remains unchanged in expectation after a jump. In Section~\ref{sect:test}, we describe a family of exponential test functions satisfying the above property, which is also rich enough to asymptotically characterize the distribution of the queue lengths.}  

It may be helpful for the reader to compare the arguments that follow to the proof of Lemma 2.3 in \cite[Appendix A.2]{Miya2017}, where the author also seeks to find test functions for which the jump terms vanish. In that paper, simultaneous events are handled by an approach that is slightly different from ours. 
To make the jump terms vanish, we now analyze the terms $f(X_{\delta_m} )- f(X_{\delta_m-})$. If the event $\delta_m$ corresponds to the $q$th external arrival to station $ i \in \cal E$, then 
\begin{displaymath}
X_{\delta_m} = X_{\delta_m-} + (e^{(i)},\ T_{e,i}(q)e^{(i)},\ 0),
\end{displaymath}
where $(e^{(i)},\ T_{e,i}(q)e^{(i)},\ 0) \in \Z^d_+ \times \R^{d} \times \R^d$. Since $T_{e, i}(q)$ is independent of $X_{\delta_m-}$ and
$U_i(q)$, we have
\begin{align}
\E_{\nu}\Big[\big(f(X_{\delta_m} )- f(X_{\delta_m-}) \big) 1(0 < U_i(q) \leq t) \Big] = 0 \label{eq:jump1_zero}
\end{align}
if
\begin{equation}
  \label{eq:cond1}
   \E \Big[f(\ell+e^{i}, u+ T_{e,i}(q)e^{(i)}, v)\Big] = f(\ell, u, v), \quad i \in \cal E,
\end{equation}
for every feasible state $(\ell, u, v)\in \Z_+^d \times \R_+^{2d}$ with $u_i=0$. Similarly, if $\delta_m$ corresponds to the $q$th departure from station $ j \in \J$, then
\begin{displaymath}
X_{\delta_m} = X_{\delta_m-} + (-e^{(j)}+ \phi^{(j)}(q),\ 0 ,\ T_{s,j}(q)e^{(j)}).
\end{displaymath}
Since $T_{s,j}(q)$ and $\phi^{(j)}(q)$ are independent of $X_{\delta_m-}$ and
$S_j(q)$, we have
\begin{align}
\E_{\nu}\Big[\big(f(X_{\delta_m} )- f(X_{\delta_m-}) \big) 1(0 < S_j(q) \leq t) \Big] = 0 \label{eq:jump2_zero}
\end{align}
if
\begin{equation}
  \label{eq:cond2}
   \E \Big[f(\ell-e^{j}+\phi^{(j)}(q), u, v+  T_{s,j}(q)e^{(j)})\Big] = f(\ell, u, v), \quad j \in \J,
\end{equation}
for every feasible state $(\ell, u, v)\in \Z_+^d \times \R_+^{2d}$ with $\ell_j>0$ and
$v_j=0$.

We summarize our analysis in the following lemma.
\begin{lemma}\label{lem:MCbar}
Assume that $X$ is positive Harris recurrent and 
$X(0)$ follows the stationary distribution of $X$.
 For any function $f\in \mathcal{D}$
satisfying conditions (\ref{eq:cond1}) and (\ref{eq:cond2}), the basic adjoint relationship \eqref{eq:hard_BAR} reduces to
\begin{align}
  \label{eq:MCbar}
 & \sum_{i\in {\cal E}}   \E\Big[ \frac{ \partial f}{\partial u_i}(X(0))\Big]
  +\sum_{j\in {\cal J}}\E\Big[\frac{ \partial f}{\partial v_j}(X(0)) \Big]\\
  &\qquad -\sum_{j\in {\cal J}}\E\Big[\frac{ \partial f}{\partial v_j}(X(0))1(L_j(0)=0) \Big] = 0. \notag
\end{align}
\end{lemma}
\begin{proof}
Apply \eqref{eq:jump1_zero}, \eqref{eq:jump2_zero}, and the definition of $\cal{A}$ in \eqref{eq:A} to \eqref{eq:hard_BAR}.
\end{proof}

\begin{remark} \blue{The idea of engineering a test function to kill the jump terms in \eqref{eq:hard_BAR}  first appeared in the work of \citet{Miya2015} for a single server queue, which is further studied for a multiserver queue with heterogeneous servers in \cite{Miya2017}.  The novel feature of the present paper is the careful consideration of tightness in Section~\ref{sect:tightness}. This issue was not present in \cite{Miya2015,Miya2017} because those papers deal with a single queue as opposed to a queueing network.} 

\end{remark}
In the next section, we describe a useful family of exponential test functions that satisfies (\ref{eq:cond1}) and (\ref{eq:cond2}).

\subsection{Exponential Test Functions for a Sequence of GJNs}
\label{sect:test}
We consider the sequence of Markov processes $\{X^{(n)} ,\ n \geq 1\}$ from Section~\ref{sect:sequence}. Recall our goal, which is to show that the MGF of $Z^{(n)}(\infty) = r_nL^{(n)}(\infty)$ asymptotically satisfies the BAR \eqref{eq:srbmbar} of the approximating SRBM. As a first step, we have already derived a BAR for $X^{(n)}$. We now describe a family of exponential test functions that satisfy conditions (\ref{eq:cond1}) and (\ref{eq:cond2}).

For all $n \in \dd{Z}_{+}$, let us first define the truncation function $g^{(n)} : \dd{R} \to\dd{R}$ as
\begin{align*}
  g^{(n)}(y) = \min(y,1/{r}_{n}), \quad y \in \R.
\end{align*}
Now for each $\theta \in \dd{R}^{d}$, $\eta \in \R^d$ and $\zeta \in \R^d$, define $f_{\theta, \eta, \zeta}^{(n)} : \Z_+^d \times \R_+^{2d} \to \dd{R}_{+}$ as
\begin{align}
 \label{eq:testf}
 f_{\theta, \eta, \zeta}^{(n)}(x) & =  e^{\br{{\theta}, {\ell}}+ \sum_{i \in \sr{E}} \eta_i g^{(n)} (u_i)  + \sum_{j \in \J} \zeta_j g^{(n)} (v_j)}, \\
& \quad \text{ for } x = (\ell,u,v) \in \Z_+^d \times \R_+^{2d}. \nonumber
\end{align}

It is not hard to verify that $f_{\theta, \eta, \zeta}^{(n)} \in \mathcal{D}$. \blue{Our in \eqref{eq:testf} is meant to resemble a moment generating function.} Indeed, if $\eta$ and  $\zeta$ were allowed to be independent of $\theta$ and if we chose $g^{(n)}(y)=y$, then this family of test functions 
 would characterize the stationary distribution of $X^{(n)}$ via its BAR \eqref{eq:hard_BAR}. However, applying \eqref{eq:hard_BAR} to $f_{\theta, \eta, \zeta}^{(n)}$ is of little practical use, because the jump terms become too complicated to work with. 
 
Instead, we want to choose $f_{\theta, \eta, \zeta}^{(n)}$ to satisfy the conditions of Lemma~\ref{lem:MCbar}, so that we can use \eqref{eq:MCbar} instead. To do so, we must choose $\eta$ and $\zeta$ as functions of $\theta$ (i.e.\ $\eta =\eta^{(n)}(\theta)$ and $\zeta = \zeta^{(n)}(\theta)$), significantly reducing the size of this family of functions.
\blue{Following the logic behind \eqref{eq:cond1} and \eqref{eq:cond2}}, we choose $\eta^{(n)}(\theta)$ and $\zeta^{(n)}(\theta)$ to satisfy
\begin{align}
\label{eq:eta 1}
 & e^{\theta_{i}} \dd{E}\big( e^{\eta_{i}^{(n)}(\theta_{i}) g^{(n)}(T_{e,i}^{(n)})} \big) = 1, \qquad i \in \sr{E},\\
\label{eq:zeta 1}
 & t_{j}(\theta) \dd{E}\big( e^{\zeta_{j}^{(n)}(\theta) g^{(n)}(T_{s,j}^{(n)})} \big) = 1,  \qquad j \in \J,
\end{align}
where
\begin{align*}
   t_{j}(\theta) = e^{-\theta_{j}} \Big( \sum_{k \in \sr{J}} p_{jk} e^{\theta_{k}} + p_{j0} \Big). 
\end{align*}
We have rewritten $\eta_{i}^{(n)}(\theta)$ as $\eta_{i}^{(n)}(\theta_{i})$ because it is independent of $\theta_{k}$ for $k \ne i$. For the remainder of the paper, we write
\begin{align*}
f_{\theta}^{(n)}(x) = f_{\theta, \eta^{(n)}(\theta), \zeta^{(n)}(\theta)}^{(n)}(x).
\end{align*} 
This reduced family of test functions can only characterize the distribution of $Z^{(n)}(\infty)$ asymptotically as $n \to \infty$, which is enough for our purposes. The following lemma is similar to Lemma 2.3 of \cite{Miya2017}, and says that $f_{\theta}^{(n)}$ satisfies the conditions of Lemma~\ref{lem:MCbar}.
\begin{lemma}
\label{lem:terminal condition 1}
\blue{There exists $M> 0$ such that for all $\theta \in \R^{d}$ with $\| \theta \| \leq M$,} the solutions  $\eta_{i}^{(n)}(\theta_{i})$ and $\zeta_{j}^{(n)}(\theta)$ to  \eq{eta 1}--\eq{zeta 1} are well defined and finite for all $n \geq 1$,  $i \in \sr{E}$ and $j \in \sr{J}$. Furthermore, conditions (\ref{eq:cond1}) and (\ref{eq:cond2}) are satisfied for $f = f_{\theta}^{(n)}$. 
\end{lemma}
\begin{proof}
Once we assume that $\eta_{i}^{(n)}(\theta_{i})$ and $\zeta_{j}^{(n)}(\theta)$ are well defined and finite, the second claim of the lemma becomes trivial to verify by \eqref{eq:eta 1} and \eqref{eq:zeta 1}. We now check that these functions are well defined for all $\theta \in \dd{R}^{d}$ with $\|\theta\| \leq M$. The argument to show that $\eta_{i}^{(n)}(\theta_{i})$ is well defined is simple, and is repeated here from Section 2.3 of \cite{Miya2017}. For $i \in \sr{E}$, let
\begin{align*}
  {G}_{e,i}^{(n)}(y) = \dd{E}\big( e^{y g^{(n)}(T_{e,i}^{(n)})} \big), \qquad y \in \dd{R},
\end{align*}
be the moment generating function of
$g^{(n)}(T_{e,i}^{(n)})$. Then ${G}_{e,i}^{(n)}(y)$ exists for
all $y \in \dd{R}$ because $g^{(n)}(T_{e,i}^{(n)})$ is a bounded random variable for every $n$. Furthermore, the inverse $({G}_{e,i}^{(n)})^{-1}$ exists as ${G}_{e,i}^{(n)}(y)$ is strictly increasing. \blue{Observe however that this inverse is only defined on the interval
\begin{align*}
\Big(\lim_{y \to -\infty} {G}_{e,i}^{(n)}(y), \lim_{y \to \infty} {G}_{e,i}^{(n)}(y) \Big) = \big(\Prob(T_{e,i}^{(n)} = 0) , \infty \big).
\end{align*}
Hence, \eq{eta 1} yields
\begin{align*}
  \eta_{i}^{(n)}(\theta_{i}) = ({G}_{e,i}^{(n)})^{-1}(e^{-\theta_{i}}), \qquad \theta_{i} \in \big(-\infty, -\log(\Prob(T_{e,i}^{(n)} = 0))\big),
\end{align*}
with the convention that $-\log(0) = \infty$. We know that there exists some  $M_{e,i} > 0$ such that 
\begin{align*}
0 < M_{e,i} < -\log(\Prob(T_{e,i}^{(n)} = 0)) , \quad n \geq 1
\end{align*}
because $\Prob(T_{e,i}^{(n)} = 0) \neq 1$ for any $n\geq 1$, and $\lim_{n \to \infty} \Prob(T_{e,i}^{(n)} = 0) \neq 1$. The former is true because $1/\E(T_{e,i}^{(n)}) = \lambda_{e,i}^{(n)} < \infty$ for each $n$, and the latter follows from \eqref{eq:momentarr} and the fact that $\lambda_{e,i} < \infty$ there. Hence, $\eta_{i}^{(n)}(\theta_{i})$ is well defined and finite for all $\theta_i \in (-\infty, M_{e,i})$.} We now derive  a similar expression  for $\zeta_j^{(n)}(\theta)$.  For $ j \in \sr{J}$ and $y \in \R$, define
\begin{align*}
 {G}_{s,j}^{(n)}(y) = \dd{E}\big( e^{y g^{(n)}(T_{s,j}^{(n)})} \big),
\end{align*} 
\blue{and observe that $({G}_{s,j}^{(n)})^{-1}$ exists because ${G}_{s,j}^{(n)}(y)$ is an increasing function. Define $\chi_{j}^{(n)}(y)$ as
\begin{align} \label{eq:chi 1}
  \chi_{j}^{(n)}(y) = ({G}_{s,j}^{(n)})^{-1}(e^{-y}), \qquad y\in \big(-\infty, -\log(\Prob(T_{s,j}^{(n)} = 0))\big),
\end{align}
and let $M_{s,j} > 0$ be such that 
\begin{align*}
 0 < M_{s,j} < -\log(\Prob(T_{s,j}^{(n)} = 0), \quad n \geq 1. 
\end{align*}
We conclude that $\zeta_{j}^{(n)}(\theta)$ is well defined and satisfies
\begin{align}
\label{eq:zeta 2}
  &\zeta_{j}^{(n)}(\theta) = \chi_{j}^{(n)}(\log t_{j}(\theta))
=({G}_{s,j}^{(n)})^{-1}\Big(\frac{1}{t_j(\theta)}\Big), \\
 & \qquad \text{ for } \theta \text{ such that }\log t_{j}(\theta) < M_{s,j},\notag
\end{align}
where $t_j(\theta)$ is defined in \eqref{eq:zeta 1}.}
\end{proof}
We now present the BAR for this special family of exponential test functions.
\begin{lemma}
\label{lem:SE 1}
Assume that $X^{(n)}$ is positive Harris recurrent, and let 
\begin{align*}
X^{(n)}(\infty) = (L^{(n)}(\infty), R_e^{(n)}(\infty), R_s^{(n)}(\infty))
\end{align*} have the stationary distribution of $X^{(n)}$. Then
\begin{align}
\label{eq:SE 1}
  & \sum_{i \in {\cal E}} \eta_{i}^{(n)}(\theta_{i}) \dd{E}\big[1(R^{(n)}_{e,i}(\infty) < 1/{r}_{n}) f_{\theta}^{(n)}(X^{(n)}(\infty))\big]  \\
 & \quad +  \sum_{j \in \J}  \zeta_{j}^{(n)}(\theta) \dd{E}\big[1(R^{(n)}_{s,j}(\infty) < 1/{r}_{n}) f_{\theta}^{(n)}(X^{(n)}(\infty))\big] \notag \\
  & \quad - \sum_{j \in \J} \zeta_{j}^{(n)}(\theta) \dd{E}\big[1(R^{(n)}_{s,j}(\infty) < 1/{r}_{n}, L_{j}^{(n)}(\infty)=0) f_{\theta}^{(n)}(X^{(n)}(\infty))\big] \notag\\
  & = 0 \notag
\end{align}
\blue{for all $\|\theta\| \leq M$, where $M > 0$ is as in Lemma~\ref{lem:terminal condition 1}.}
\end{lemma}
\begin{proof}
This lemma follows immediately from applying Lemmas~\ref{lem:MCbar} and
  \lemt{terminal condition 1} to the test function $f_{\theta}^{(n)}(x)$ from \eqref{eq:testf}, because its right side partial derivatives are
\begin{align*}
 \frac {\partial}{\partial u_i} f_{\theta}^{(n)}(x) = &\  \eta_{i}^{(n)}(\theta_{i}) 1(u_i < 1/{r}_{n}) f_{\theta}^{(n)}(x), \quad i \in \sr{E},\\
\frac {\partial}{\partial v_j} f_{\theta}^{(n)}(x) = &\ \zeta_{j}^{(n)}(\theta) 1(v_j < 1/{r}_{n}) f_{\theta}^{(n)}(x), \quad j \in \J.
\end{align*}
\end{proof}
In order for Lemma~\ref{lem:SE 1} to be of practical use, we need to know the behavior of $\eta_{i}^{(n)}(\theta_{i})$ and $\zeta_{j}^{(n)}(\theta)$; however, these functions are defined implicitly. In the next section, we use quadratic approximations of $\eta_{i}^{(n)}(\theta_{i})$ and $\zeta_{j}^{(n)}(\theta)$ to convert \eqref{eq:SE 1} into a more convenient expression that resembles \eqref{eq:srbmbar}.
\section{Approximate BAR} \label{SECT:PRELIMIT}
\setnewcounter
This section is devoted to proving Proposition~\ref{lem:prelimitMGFBAR}, which we now state. 

\begin{proposition} \label{lem:prelimitMGFBAR}
Assume that all conditions stated in Theorem~\ref{thr:heavy traffic 1} are satisfied. Recall the definitions of $\gamma(\theta)$ and $\gamma_{j}(\theta)$ from \eq{gamma 1}. For $ j \in \J$ and $\theta \leq 0$, define
\begin{align}\label{eq:laplaceprelimit}
 & \varphi^{(n)}(\theta) = \E [e^{\langle\theta, Z^{(n)}(\infty)\rangle}] ,\\
 & \varphi_{j}^{(n)}(\theta) = \E [e^{\langle\theta, Z^{(n)}(\infty)\rangle} | Z_{j}^{(n)}(\infty)=0]. \notag
\end{align}
Let 
\begin{align}
\epsilon^{(n)}(\theta) = \gamma(\theta) \varphi^{(n)}(\theta)  + \sum_{j \in \J} b_{j} \gamma_j(\theta) \varphi_{j}^{(n)}(\theta) \quad n\ge 1,\ \theta \leq 0, \label{eq:prelimBAR} 
\end{align}
then 
\begin{align*}
\lim_{n \to \infty} \sup \limits_{\substack{\theta < 0 \\ 0 < \|\theta\| \leq M}} \frac{\abs{\epsilon^{(n)}(\theta)}}{\|\theta\|} =0,
\end{align*}
\blue{where $M > 0$ is as in Lemma~\ref{lem:terminal condition 1}.}
\end{proposition} 
This result states that the steady state distributions of the queue lengths asymptotically satisfy \eqref{eq:srbmbar} as $n \to \infty$. It plays an essential role in the proof of Theorem~\ref{thr:heavy traffic 1} in \sectn{proofs}.
 
The idea behind the proof is to use \eqref{eq:SE 1} as a starting point, and use quadratic approximations of $\eta_{i}^{(n)}(\theta_{i})$ and $\zeta_{j}^{(n)}(\theta)$ to arrive at \eqref{eq:prelimBAR}. In \sectn{taylor_exp} we use Taylor expansion to obtain the quadratic approximations of $\eta_{i}^{(n)}(\theta_{i})$ and $\zeta_{j}^{(n)}(\theta)$. We then show that under heavy traffic scaling, the associated approximation error vanishes in an appropriate fashion. We prove \lem{prelimitMGFBAR} in \sectn{MGFBAR}.

\subsection{Taylor Expansions}
\label{sect:taylor_exp}

We begin with a general lemma, which describes the behavior of functions that are implicitly defined in a manner similar to $\eta_{i}^{(n)}(\theta_{i})$ in \eqref{eq:eta 1}. The lemma is similar to Lemma 2.4 of \cite{Miya2017}.

\begin{lemma}
\label{lem:concave 1}
Let $H$ be a bounded, non-negative random variable with $\E H > 0$ and set
\begin{align*}
\lambda_H = \frac{1}{\E H}, \quad \sigma_H^2 = \mathrm{ Var} (H).
\end{align*}
Then the function $f(x): \R \to \R$ satisfying
\begin{align*}
\E (e^{f(x) H}) = e^{-x}, \quad x \in \blue{\big(-\infty, -\log(\Prob(H = 0)) \big)},
\end{align*}
is well defined and finite, with the convention that $-\log(0) = \infty$. Furthermore,
\begin{enumerate}[(a)]
\item $f(x)$ \blue{is infinitely differentiable.} \label{item:concave1_b}
\item $f(x)$ is decreasing and concave. \label{item:concave1_c}
\item For any \blue{$K \in (0, -\log(\Prob(H = 0))$}, $f(x)$ satisfies\label{item:concave1_d}
\begin{align}
\label{eq:generic asymptotic 1}
& \Big|f(x) + \lambda_H x + \frac 12 \lambda_H^{3} {\sigma}_{H}^{2} x^{2} \Big| \le c_{H}(x),\\
\label{eq:generic asymptotic 0}
 &|f(x)| \le \max\{\lambda_H, \abs{f(K)}/K\} |x|, \quad \abs{x} \leq K,
\end{align}
where 
\begin{align} \label{eq:generic_error}
c_{H}(x) =  \frac{x^{2} }{2} \sup_{|y|< \abs{x}} |f''(y) - f''(0)|.
\end{align}
\end{enumerate}
\end{lemma}

The proof of this lemma is deferred to \app{concave 1}. We now apply it to obtain expansions of $\eta_{i}^{(n)}(\theta_{i})$ and $\zeta_{j}^{(n)}(\theta)$. \blue{Let $M> 0 $ be as in Lemma~\ref{lem:terminal condition 1}.} Set
\begin{align*}
 & \frac{1}{\widetilde{\lambda}_{e,i}^{(n)}} = \dd{E}(g^{(n)}(T_{e,i}^{(n)})), \quad \big(\widetilde{\sigma}_{e,i}^{(n)}\big)^{2} = \mathrm{ Var}(g^{(n)}(T_{e,i}^{(n)})), \quad i \in \sr{E}, \\
 & \frac{1}{\widetilde{\lambda}_{s,j}^{(n)}} = \dd{E}(g^{(n)}(T_{s,j}^{(n)})), \quad \big(\widetilde{\sigma}_{s,j}^{(n)}\big)^{2} = \mathrm{ Var}(g^{(n)}(T_{s,j}^{(n)})), \quad j \in \sr{J}.
\end{align*} 
Observe that the uniform integrability assumptions \eqref{eq:uniarr} and \eqref{eq:uniser}, together with the definition of $g^{(n)}(x)$, imply that for all $i \in \mathcal{E}$ and $j \in \J$, 
\begin{align}
\widetilde{\lambda}_{e,i}^{(n)} \to \lambda_{e,i}, \quad \widetilde{\lambda}_{s,j}^{(n)} \to \lambda_{s,j}, \quad \widetilde{\sigma}_{e,i}^{(n)} \to \sigma_{e,i}, \quad \text{ and } \quad \widetilde{\sigma}_{s,j}^{(n)} \to \sigma_{s,j} \label{eq:tildes_converge}
\end{align} 
(cf. \eqref{eq:momentarr}, \eqref{eq:momentser} and \eqref{eq:lamdaconverge}). 
Lemma~\ref{lem:concave 1} now trivially applies to $\eta_{i}^{(n)}(\theta_{i})$ to obtain 
\begin{align}
\label{eq:eta asymptotic 1}
& \Big|\eta^{(n)}_i(\theta_i) + \widetilde{\lambda}^{(n)}_{e,i} \theta_i + \frac 12 \big(\widetilde{\lambda}^{(n)}_{e,i}\big)^{3} \big(\widetilde{\sigma}^{(n)}_{e,i}\big)^{2} \theta_i^{2} \Big| \le c^{(n)}_{e,i}(\theta_i),\\
\label{eq:eta asymptotic 0}
 &|\eta^{(n)}_i(\theta_i)| \le \hat{c}_{e,i}^{(n)}(M) |\theta_i| , \quad \abs{\theta_i} \leq M,
\end{align}
where $c^{(n)}_{e,i}$ corresponds to $c_H$ in \eqref{eq:generic_error} and for any \blue{$K \in (0,M]$, we define $\hat{c}_{e,i}^{(n)}(K)$ as}
\begin{align*}
\hat{c}_{e,i}^{(n)}(K) =  \max \big{\{}\widetilde{\lambda}^{(n)}_{e,i}, \abs{\eta^{(n)}_i(K)}/K \big{\}}.
\end{align*} 
Recall that $\zeta^{(n)}_{j}(\theta) = \chi^{(n)}_{j}(\log t_j(\theta))$, where $\chi^{(n)}_{j}$ is defined in \eqref{eq:chi 1} and 
\begin{align*}
\log t_j(\theta) = -\theta_j + \log \Big( \sum_{k \in \J} p_{jk}e^{\theta_k} + p_{j0} \Big) \in C^{\infty}(\R^d).
\end{align*}

By Taylor expansion, one can verify that $\log t_{j}(\theta)$ satisfies
\begin{align}\label{eq:t 1}
 & \Big|\log t_{j}(\theta) - \Big( - \theta_{j} + \sum_{k \in \sr{J}} p_{jk}\theta_{k}\Big) \\
 & \hspace{10ex} + \frac 12 \Big(\sum_{k \in \sr{J}}\theta_k^2 p_{jk} - \big( \sum_{k \in \sr{J}}\theta_k p_{jk} \big)^2 \Big)\Big| \le c_{1,j}(\theta), \notag \\
 & \Big|\log^{2} t_{j}(\theta) - \Big( - \theta_{j} + \sum_{k \in \sr{J}} p_{jk} \theta_{k} \Big)^{2} \Big| \le c_{2,j}(\theta), \notag
\end{align}
where
\begin{align}\label{eq:conetwo}
  & c_{1,j}(\theta) = \frac 16 \sup_{\|y\| \leq \|\theta\|} \Bigg| \sum_{k,l,m \in \sr{J}} \theta_k \theta_l \theta_m  \frac{\partial^3 \log t_j}{\partial \theta_k \partial \theta_l \partial \theta_m}(y)  \Bigg|,  \\
& c_{2,j}(\theta) = \frac 16 \sup_{\|y\| \leq \|\theta\|} \Bigg| \sum_{k,l,m \in \sr{J}} \theta_k \theta_l \theta_m  \frac{\partial^3 \log^2 t_j}{\partial \theta_k \partial \theta_l \partial \theta_m}(y)  \Bigg|. \notag
\end{align}
Note that both $c_{1,j}(\theta)$ and $c_{2,j}(\theta)$ are finite for each $\theta$ and $j \in \J$ because $\log t_j(\theta)$ belongs to $C^{\infty}(\R^d)$.

Applying Lemma~\ref{lem:concave 1} to $\chi^{(n)}_j$, together with the Taylor expansions above, we have the following result about $\zeta^{(n)}_j(\theta)$. 
For each $j \in \J$  and $n \geq 1$, 
\begin{align}
\label{eq:zeta asymptotic 1}
 & \Big|\zeta^{(n)}_{j}(\theta) + \widetilde{\lambda}^{(n)}_{s,j} \Big( - \theta_{j} + \sum_{k \in \sr{J}} p_{jk} \theta_{k} \Big) \\
 & \hspace{10ex}+ \frac{1}{2}\widetilde{\lambda}^{(n)}_{s,j} \Big(\sum_{k \in \sr{J}}\theta_k^2 p_{jk} - \big( \sum_{k \in \sr{J}}\theta_k p_{jk} \big)^2 \Big) \notag \\
 & \hspace{10ex} + \frac 12 \big(\widetilde{\lambda}^{(n)}_{s,j}\big)^{3} \big(\widetilde{\sigma}^{(n)}_{s,j}\big)^{2} \Big(- \theta_{j} + \sum_{k \in \sr{J}} p_{jk} \theta_{k} \Big)^{2}  \Big| \le c^{(n)}_{s,j}(\theta) ,\nonumber\\
\label{eq:zeta asymptotic 0}
 & |\zeta^{(n)}_{j}(\theta)| \le \hat{c}_{s,j}^{(n)}(M) \|\theta\|, \quad \|\theta\| \leq M,
\end{align}
where
\begin{align}\label{eq:err_zeta}
c^{(n)}_{s,j}(\theta) = \sup_{\abs{y} \leq \abs{\log t_j(\theta)}}& \frac 12 \log^2 t_j(\theta) \Big| \big(\chi^{(n)}_j\big)''(y) - \big(\chi^{(n)}_j\big)''(0)\Big| \\
&+ \widetilde{\lambda}^{(n)}_{s,j}c_{1,j}(\theta) +  \frac 12 \big(\widetilde{\lambda}^{(n)}_{s,j}\big)^{3} \big(\widetilde{\sigma}^{(n)}_{s,j}\big)^{2}c_{2,j}(\theta), \notag
\end{align}
$c_{1,j}(\theta)$ and $c_{2,j}(\theta)$ are as in \eqref{eq:conetwo}, \blue{and for any $K \in (0,M]$, }
\begin{align}\label{eq:clip_def} 
&\hat{c}_{s,j}^{(n)}(K) = c_{Lip,j}(K)  \max \big{\{}\widetilde{\lambda}^{(n)}_{s,j}, \abs{\chi^{(n)}_j(\widehat{K})}/\widehat{K} \big{\}}, \\
& \widehat{K} = \sup_{\|\theta\| \leq K} \abs{\log t_j(\theta)}, \quad \text{ and } \quad c_{Lip,j}(K) = \sup_{0 < \|\theta\| \leq K} \frac{\abs{\log t_j(\theta)}}{\|\theta\|}. \notag
\end{align}
We know $c_{Lip,j}(K)< \infty$ because $\log t_j(\theta) \in C^{\infty}(\R^d)$ and is therefore is locally Lipschitz.
\subsection{Error bounds}
\label{sect:error bounds}

In order for the quadratic approximations of $\eta_{i}^{(n)}(\theta_{i})$ and $\zeta_{j}^{(n)}(\theta)$ to be useful, we need the error bounds $c^{(n)}_{e,i}(\theta)$ and $c^{(n)}_{s,j}(\theta)$ to be small. Recall that \eqref{eq:SE 1} is a statement about the unscaled vector $X^{(n)}(\infty)$, but \eqref{eq:prelimBAR} deals with the scaled queue length vector $Z^{(n)}(\infty) = r_n L^{(n)}(\infty)$. From the form of the test function in \eqref{eq:testf}, we see that by replacing $\theta$ by $r_n \theta$, \eqref{eq:SE 1} becomes a statement about the scaled queue length $Z^{(n)}(\infty)$. Under the heavy traffic scaling, the errors from the quadratic approximations vanish asymptotically in neighborhoods of the origin. The following lemma presents this statement formally.
\begin{lemma}
\label{lem:uniform 1}
\blue{Let $M> 0 $ be as in Lemma~\ref{lem:terminal condition 1}. For any $K > 0$ such that $r_n K \leq M$, }
\begin{align} \label{eq:uniform_eta}
\lim_{n \to \infty} \sup_{0 < |\theta_i|<  K} \frac{c^{(n)}_{e,i}(r_n \theta_i)}{r_n^2 \theta_i^2}  =0, \quad i \in \sr{E},
\end{align}
and
\begin{align}\label{eq:uniform_zeta}
\lim_{n \to \infty} \sup_{0 < \|\theta\|<  K} \frac{c^{(n)}_{s,j}(r_n \theta)}{r_n^2 \|\theta\|^2}  =0, \quad j \in \J.
\end{align}
\end{lemma}
This lemma is proved in Section~\appt{uniform1_gen_proof}. Our next result states that the functions $\eta^{(n)}_i(\theta_i)$ and $\zeta_j^{(n)}(\theta)$ are locally Lipschitz in small neighborhoods of the origin, with Lipschitz constants that do not depend on $n$. Its proof is postponed to Section~\ref{app:uniform_slopes}. 
\begin{lemma} \label{lem:uniform_slopes}
\blue{Let $M> 0 $ be as in Lemma~\ref{lem:terminal condition 1}.} For $i \in \sr{E}$ and $j \in \J$, \blue{and $K \in (0, r_n M]$, }
\begin{align*} 
\sup_{n \geq 1} \hat{c}_{e,i}^{(n)}(r_nK) < \infty \quad \text{ and } \quad \sup_{n \geq 1} \hat{c}_{s,j}^{(n)}( r_nK) < \infty.
\end{align*}
\end{lemma}
In the rest of this section, we will frequently use the following bound. Recall the definition of  $f_{r_n\theta}^{(n)}(x)$ from \eqref{eq:testf}, where $x = (\ell,u,v) \in \Z^d_+ \times \R^{2d}_+$. Using \eqref{eq:eta asymptotic 0}, \eqref{eq:zeta asymptotic 0}, and Lemma~\ref{lem:uniform_slopes}, it follows that for any $K \in (0,M]$,  we can define $c_f(K)$ to satisfy
\begin{align} \label{eq:testf_bound}
 &\sup_{n \geq 1} \sup_{\substack{\theta < 0 \\ \|\theta\| \leq K}}  f_{r_n \theta}^{(n)}(x) \\
 & \quad = \sup_{n \geq 1} \sup_{\substack{\theta < 0 \\ \|\theta\| \leq K}} e^{\br{{r_n\theta}, {\ell}}+ \sum_{i \in \sr{E}} \eta_i^{(n)}(r_n\theta_i) g^{(n)} (u_i)  + \sum_{j \in \J} \zeta_j^{(n)}(r_n \theta) g^{(n)} (v_j)} \notag \\
 & \quad \leq c_f(K) \equiv  \sup_{n \geq 1}e^{\sum_{i \in \sr{E}} \hat{c}_{e,i}^{(n)}(r_nK) K +\sum_{j \in \J} \hat{c}_{s,j}^{(n)}(r_nK) K}  < \infty. \notag
\end{align}
This bound holds for all $x \in \Z^d_+ \times \R^{2d}_+$.
We now state several lemmas that we will use to prove Proposition~\ref{lem:prelimitMGFBAR}.
\begin{lemma}
\label{lem:boundary_probs}
Recall the heavy traffic condition \eqref{eq:heavy 1}. For any station $j \in \J$,
\begin{align*}
\dd{P}(L_{j}^{(n)}(\infty)=0) = 1 - \lambda^{(n)}_{a,j}/\lambda^{(n)}_{s,j} = r_{n} b_{j}/\lambda^{(n)}_{s,j}.
\end{align*}
\end{lemma}
\begin{lemma}
\label{lem:residualUI}
The sequences $\{{R}^{(n)}_{e,i}(\infty), n \geq 1 \}$ and $\{{R}^{(n)}_{s,j}(\infty), n \geq 1 \}$ are uniformly integrable for all $i \in \sr{E}$ and $j \in \J$.
\end{lemma}
We would like to point out that Lemmas~\ref{lem:boundary_probs} and \ref{lem:residualUI} are not novel, and can be proved using the well developed theory of Palm calculus (see for example \cite[Chapter 1]{BaccBrem2003} or \cite[Chapter 4]{BoucDijk2011}). However, to keep this paper self-contained, we avoid using Palm calculus and prove these lemmas in Section~\appt{bp+ru} using the BAR \eqref{eq:hard_BAR}.

\begin{lemma}
\label{lem:limiting 1}
\blue{Let $M> 0 $ be as in Lemma~\ref{lem:terminal condition 1}, and $K \in (0, r_n M]$.} The following statements are true:
\begin{align}\label{eq:limiting 1}
 & \lim_{n \to \infty} \sup_{\substack{\theta < 0 \\ 0 < \|\theta\| \leq K}} \frac{1}{\|\theta\|} \Big|  -\frac{1}{r_n^2}\Big( \sum_{i \in {\cal E}} \eta_{i}^{(n)}(r_n \theta_{i}) + \sum_{j \in \J} \zeta_{j}^{(n)}(r_n \theta)\Big) - \gamma(\theta) \Big|\\
 & \quad = 0, \notag\\
 & \lim_{n \to \infty} \sup_{\substack{\theta < 0 \\ \|\theta\| \leq K}}  \Big| \E \big[f_{r_n \theta}^{(n)}(X^{(n)}(\infty))\big]  - \varphi^{(n)}(\theta) \Big| = 0, \label{eq:limiting 2} \\
& \lim_{n \to \infty}  \sup_{\substack{\theta < 0 \\  \|\theta\| \leq K}} \frac{1}{r_n} \Big| \E \big[1( L_{j}^{(n)}(\infty)=0)\big(f_{r_n \theta}^{(n)}(X^{(n)}(\infty))- e^{\br{\theta, {Z}^{(n)}(\infty)}}\big)\big] \Big|\label{eq:limiting 3} \\
& \quad = 0, \qquad j \in \J. \notag 
\end{align}
\end{lemma}
The proof of this lemma is postponed until  Section~\appt{limiting 1}.

\subsection{Proof of \lem{prelimitMGFBAR}}
\label{sect:MGFBAR}
We now prove Proposition~\ref{lem:prelimitMGFBAR} using the auxiliary lemmas stated in the previous section. As a starting point, we multiply \eqref{eq:SE 1} by $-1/r_n^2$ to see that
\begin{align}
\label{eq:fstep}
0&= -\frac{1}{r_n^2} \dd{E}\big[f_{r_n \theta}^{(n)}(X^{(n)}(\infty))\big]\bigg(\sum_{i \in {\cal E}} \eta_{i}^{(n)}(r_n \theta_{i}) + \sum_{j \in \J} \zeta_{j}^{(n)}(r_n \theta) \bigg) \\
 &\quad + \frac{1}{r_n^2} \sum_{j \in \J} \zeta_{j}^{(n)}(r_n \theta) \dd{E}\big[1( L_{j}^{(n)}(\infty)=0) f_{r_n \theta}^{(n)}(X^{(n)}(\infty))\big]\notag \\ 
 & \quad  - \frac{1}{r_n^2}\sum_{j \in \J} \zeta_{j}^{(n)}(r_n \theta) 
\notag\\
& \qquad \qquad \times \dd{E}\big[1(R^{(n)}_{s,j}(\infty) \geq 1/{r}_{n}, L_{j}^{(n)}(\infty)=0) f_{r_n \theta}^{(n)}(X^{(n)}(\infty))\big] \notag \\ 
 & \quad  + \frac{1}{r_n^2}\sum_{i \in {\cal E}}  \eta_{i}^{(n)}(r_n \theta_{i}) \dd{E}\big[1(R^{(n)}_{e,i}(\infty) \geq 1/{r}_{n}) f_{r_n \theta}^{(n)}(X^{(n)}(\infty))\big] \notag  \\ 
&  \quad  +  \frac{1}{r_n^2}\sum_{j \in \J} \zeta_{j}^{(n)}(r_n \theta) \dd{E}\big[1(R^{(n)}_{s,j}(\infty) \geq  1/{r}_{n}) f_{r_n \theta}^{(n)}(X^{(n)}(\infty))\big]. \notag
\end{align}

We claim that the last three lines are negligible. For this, we wish to show that 
\begin{align*}
 &\lim_{n \to \infty} \sup_{\substack{\theta < 0 \\ 0 < \|\theta\| \leq M}}\ \frac{1}{\|\theta\|}  \frac{1}{r_n^2}\Big|\sum_{j \in \J} \zeta_{j}^{(n)}(r_n \theta) \dd{E}\big[1(R^{(n)}_{s,j}(\infty) \geq 1/{r}_{n}) f_{r_n \theta}^{(n)}(X^{(n)}(\infty))\big] \Big| \\
 &\quad = 0,
\end{align*}
then similar statements hold for the remaining two lines. Observe that for each $j \in \J$,
\begin{align*}
&\sup_{\substack{\theta < 0 \\ 0 < \|\theta\| \leq M}} \  \frac{1}{r_n^2 \|\theta\|} \Big| \zeta_{j}^{(n)}(r_n \theta) \dd{E}\big[1(R^{(n)}_{s,j}(\infty) \geq 1/{r}_{n}) f_{r_n \theta}^{(n)}(X^{(n)}(\infty))\big] \Big| \\ 
& \qquad \leq \hat{c}^{(n)}_{s,j}(r_nM) c_f(M)\frac{1}{r_n} \dd{P}(R^{(n)}_{s,j}(\infty) \geq 1/{r}_{n})
\end{align*}
by \eq{zeta asymptotic 0} and \eqref{eq:testf_bound}. By Lemma~\ref{lem:uniform_slopes} and \lem{residualUI}, this upper bound vanishes as $n \to \infty$.
Thus, we have succeeded in proving that
\begin{align*}
& \lim_{n \to \infty} \sup_{\substack{\theta < 0 \\ 0 < \|\theta\| \leq M}}\ \frac{1}{\|\theta\|} \bigg|\sum_{j \in \J} \frac{1}{r_n^2}\zeta_{j}^{(n)}(r_n \theta) \E \big[f_{r_n \theta}^{(n)}(X^{(n)}(\infty)) 1(L_{j}^{(n)}(\infty)=0)\big] \notag \\
  & \qquad  -\frac{1}{r_n^2}  \Big( \sum_{i \in {\cal E}}  \eta_{i}^{(n)}(r_n \theta_{i}) + \sum_{j \in \J} \zeta_{j}^{(n)}(r_n \theta)\Big) \E \big[f_{r_n \theta}^{(n)}(X^{(n)}(\infty))\big] \bigg|\\
  &\quad =  0.
\end{align*}
For the next step, we apply \eqref{eq:testf_bound} and \lem{limiting 1} to see that 
\begin{align*}
&\ \lim_{n \to \infty} \sup_{\substack{\theta < 0 \\ 0 < \|\theta\| \leq M}} \frac{1}{\|\theta\|} \bigg|-\frac{1}{r_n^2}\Big( \sum_{i \in {\cal E}}  \eta_{i}^{(n)}(r_n \theta_{i}) + \sum_{j \in \J} \zeta_{j}^{(n)}(r_n \theta)\Big) \\
& \qquad\qquad \times \E \big[f_{r_n \theta}^{(n)}(X^{(n)}(\infty))\big] - \gamma(\theta)\varphi^{(n)}(\theta)\bigg| \\
&\quad \leq\ \lim_{n \to \infty} \sup_{\substack{\theta < 0 \\ 0 < \|\theta\| \leq M}} \frac{1}{\|\theta\|} \bigg| -\frac{1}{r_n^2}\Big( \sum_{i \in {\cal E}}  \eta_{i}^{(n)}(r_n \theta_{i}) + \sum_{j \in \J} \zeta_{j}^{(n)}(r_n \theta)\Big) - \gamma(\theta) \bigg| \\
&\qquad\qquad \times \E \big[f_{r_n \theta}^{(n)}(X^{(n)}(\infty))\big] \\
&\qquad +\ \lim_{n \to \infty} \sup_{\substack{\theta < 0 \\ 0 < \|\theta\| \leq M}} \frac{\abs{\gamma(\theta) }}{\|\theta\|} \Big|\E \big[f_{r_n \theta}^{(n)}(X^{(n)}(\infty))\big] - \varphi^{(n)}(\theta)\Big| = 0.
\end{align*}
We arrive at the intermediate result
\begin{align}
&\lim_{n \to \infty} \sup_{\substack{\theta < 0 \\ 0 < \|\theta\| \leq M}}  \frac{1}{\|\theta\|} \bigg| \gamma(\theta) \varphi^{(n)}(\theta)  \label{eq:intermed_result} \\
& \qquad +  \sum_{j \in \J} \frac{1}{r_n^2}\zeta_{j}^{(n)}(r_n \theta) \E \big[f_{r_n \theta}^{(n)}(X^{(n)}(\infty))1(L_{j}^{(n)}(\infty)=0)\big] \bigg| = 0. \notag 
\end{align}
Recall the definition of $\varphi^{(n)}_j(\theta)$ from \eqref{eq:laplaceprelimit} and use the telescoping sum
\begin{align*}
& \lim_{n \to \infty} \sup_{\substack{\theta < 0 \\ 0 < \|\theta\| \leq M}}  \frac{1}{\|\theta\|}\Big| \frac{1}{r_n^2}\zeta_{j}^{(n)}(r_n \theta) \E \big[f_{r_n \theta}^{(n)}(X^{(n)}(\infty)) 1(L_{j}^{(n)}(\infty)=0)\big] \\
& \qquad \qquad -b_j \gamma_j(\theta) \varphi^{(n)}_j(\theta) \Big| \\
& \quad \leq \ \lim_{n \to \infty} \sup_{\substack{\theta < 0 \\ 0 < \|\theta\| \leq M}}  \frac{1}{\|\theta\|}\Big| \frac{1}{r_n^2}\zeta_{j}^{(n)}(r_n \theta) -\frac{b_j}{P(L_{j}^{(n)}(\infty)=0)} \gamma_j(\theta)\Big|\\
& \qquad \qquad \times\E \big[f_{r_n \theta}^{(n)}(X^{(n)}(\infty)) 1(L_{j}^{(n)}(\infty)=0)\big] \\
& \qquad + \lim_{n \to \infty} \sup_{\substack{\theta < 0 \\ 0 < \|\theta\| \leq M}}  \frac{1}{\|\theta\|}\frac{b_j}{P(L_{j}^{(n)}(\infty)=0)} \gamma_j(\theta) \\
& \qquad\qquad \times\Big|\E \Big[\big[f_{r_n \theta}^{(n)}(X^{(n)}(\infty)) - e^{\br{\theta,Z^{(n)}(\infty)}}\big] 1(L_{j}^{(n)}(\infty)=0)\Big] \Big|.
\end{align*}
Using \eqref{eq:intermed_result} we see that to complete the proof of Proposition~\ref{lem:prelimitMGFBAR}, all we need to do is show the upper bound above equals zero. To show that the first term is zero, we recall from Lemma~\ref{lem:boundary_probs} that
\begin{align*}
  \dd{P}(L_{j}^{(n)}(\infty)=0) = 1 - \lambda^{(n)}_{a,j}/\lambda^{(n)}_{s,j} = r_{n} b_{j}/\lambda^{(n)}_{s,j}, \quad j \in \J.
\end{align*}
Recalling the form of $\gamma_j(\theta)$ from \eqref{eq:gamma 1} and the bound $c_f(M)$ from \eqref{eq:testf_bound}, we see that
\begin{align*}
& \lim_{n \to \infty} \sup_{\substack{\theta < 0 \\ 0 < \|\theta\| \leq M}}\frac{1}{\|\theta\|}\Big| \frac{1}{r_n^2}\zeta_{j}^{(n)}(r_n \theta) -\frac{b_j}{\dd{P}(L_{j}^{(n)}(\infty)=0)} \gamma_j(\theta)\Big|\\
& \qquad \qquad  \times  \E \big[f_{r_n \theta}^{(n)}(X^{(n)}(\infty)) 1(L_{j}^{(n)}(\infty)=0)\big]\\
&\quad  = \lim_{n \to \infty} \sup_{\substack{\theta < 0 \\ 0 < \|\theta\| \leq M}}\frac{1}{\|\theta\|}\Big| \frac{\dd{P}(L_{j}^{(n)}(\infty)=0)}{r_n^2}\zeta_{j}^{(n)}(r_n \theta) -b_j \gamma_j(\theta)\Big|\\
& \qquad \qquad \times \E \big[f_{r_n \theta}^{(n)}(X^{(n)}(\infty))\big| L_{j}^{(n)}(\infty)=0\big] \\
& \quad \leq \ \lim_{n \to \infty} \sup_{\substack{\theta < 0 \\ 0 < \|\theta\| \leq M}}\frac{1}{\|\theta\|}\Big|\frac{b_j}{r_n \lambda^{(n)}_{s,j}}\zeta_{j}^{(n)}(r_n \theta)   -b_j \br{R^{(j)}, \theta}\Big|c_f(M).
\end{align*}
Furthermore,
\begin{align*}
& \lim_{n \to \infty} \sup_{\substack{\theta < 0 \\ 0 < \|\theta\| \leq M}}\frac{1}{\|\theta\|}\Big|\frac{b_j}{r_n \lambda^{(n)}_{s,j}}\zeta_{j}^{(n)}(r_n \theta)   -b_j \br{R^{(j)}, \theta}\Big|c_f(M)\\
&\quad  = \  \lim_{n \to \infty} \sup_{\substack{\theta < 0 \\ 0 < \|\theta\| \leq M}}  \frac{1}{\|\theta\|}\Big| \frac{\widetilde{\lambda}^{(n)}_{s,j}}{\lambda^{(n)}_{s,j}} b_j\br{R^{(j)},\theta}  -b_j \br{R^{(j)}, \theta}\Big|c_f(M) = 0,
\end{align*}
where the first equality is justified by \eqref{eq:lamdaconverge} together with the approximation of $\zeta_j^{(n)}(r_n\theta)$ from \eqref{eq:zeta asymptotic 1} and Lemma~\ref{lem:uniform 1}, and the second equality follows from \eqref{eq:lamdaconverge} and \eqref{eq:tildes_converge}. Now for the second term, Lemmas~\ref{lem:boundary_probs} and \ref{lem:limiting 1} tell us that for $j \in \J$,
\begin{align*}
\lim_{n \to \infty}  \sup_{\substack{\theta < 0 \\ \|\theta\| \leq M}} \frac{\Big| \E \big[1( L_{j}^{(n)}(\infty)=0)\big(f_{r_n \theta}^{(n)}(X^{(n)}(\infty))- e^{\br{\theta, {Z}^{(n)}(\infty)}}\big)\big] \Big|}{\dd{P}(L_{j}^{(n)}(\infty)=0)}  = 0.
\end{align*}
This concludes the proof of Proposition~\ref{lem:prelimitMGFBAR}.

\section{Tightness of Stationary Distributions: an Essential \blue{Proposition}}
\label{sect:tightness}
\setnewcounter 
This section is centered around the statement and proof of Proposition~\ref{lem:int_geq_boundary}, which is critical to proving that the sequence $\{Z^{(n)}(\infty),\ n \geq 1\}$ is tight. The tightness argument itself is provided as a part the proof of Theorem \ref{thr:heavy traffic 1} in \sectn{proofs}, and relies on both Propositions~\ref{lem:prelimitMGFBAR} and \ref{lem:int_geq_boundary}. We now motivate the proposition and introduce some notation needed to state it.

Let $\cal C$ be the class of functions $f$ that satisfy
\begin{enumerate}[(a)]
\item $f: \{\theta\le 0: \theta\in \R^d\} \to [0, 1]$
\item $f$ is continuous on $\{\theta< 0: \theta\in \R^d\}$
\item $\theta_1 \leq \theta_2 \Rightarrow f(\theta_1) \leq f(\theta_2)$,
\item $f(0)=1$.
\end{enumerate} 
Clearly, the MGF of any probability measure
on $\R^d_+$ belongs to ${\cal C}$. Suppose that $f$ is a pointwise
limit of a sequence of MGFs of probability measures. Then, $f\in
{\cal C}$. Such a pointwise limit is not necessarily the MGF of a
probability measure; for example, this happens when the sequence of measures is not tight. One can prove that the pointwise limit $f$ is an MGF of a probability measure if and only if $f$ is left continuous at $0$; see Lemma~\ref{LEM:LAPLACETIGHTNESS} in \sectn{proofs}. By left continuity, we mean
\begin{equation*}
  \label{eq:leftlimit}
  \lim_{\theta \uparrow 0} f(\theta) = 1,
\end{equation*}
where $\theta \uparrow 0$ means that $\theta \in \dd{R}^{d}$
approaches $0$ from left in arbitrary way.

In the tightness argument of Section~\ref{sect:proofs}, we deal with the sequence of MGFs corresponding to $\{Z^{(n)}(\infty),\ n \geq 1\}$. Loosely speaking, Proposition~\ref{lem:prelimitMGFBAR} tells us that the pointwise limit of every convergent subsequence of the MGFs satisfies the BAR of the SRBM \eqref{eq:srbmbar}. Proposition~\ref{lem:int_geq_boundary} then leverages the structure of this BAR to prove that the pointwise limit of the MGFs is left continuous at the origin, thereby proving tightness of the corresponding sequence of probability measures. This argument is made precise in Section~\ref{sect:proofs}. Having sufficiently motivated the necessity of Proposition~\ref{lem:int_geq_boundary}, we now continue with its setup. 

For our purposes, it will be beneficial
to take limits as $\theta \uparrow 0 $ along rays stemming from the origin. Each ray corresponds to some direction vector $\vc{c}_{A} = (c_1, \ldots ,
c_d)^{\rs{T}}$, where  $A \subset \J$ and  $c_i = 0$ when $i \notin A$. We write $\vc{c}_{A} > 0 $
if $c_i > 0$ when $i \in A$. For a fixed ray $\vc{c}_{A}>0$, we
consider the limits
\begin{align*}
 \lim_{\alpha \uparrow 0} f(\alpha\vc{c}_{A}), \quad \alpha \in \R,
\end{align*}
which always exist by the monotonicity of $f$. We note the following fact: 
\begin{lemma}
\label{lem:ray independent}
Fix $A \subset \J$ and consider any two direction vectors $\vc{c}_{A} >0 $ and $\tilde{\vc{c}}_{A} >0$.  For any function $f \in \cal C$,  
\begin{align*}
\lim_{\alpha \uparrow 0} f(\alpha\vc{c}_{A}) = \lim_{\alpha \uparrow 0} f(\alpha\tilde{\vc{c}}_{A}).
\end{align*}
\end{lemma}
\begin{proof}
Since $\vc{c}_{A} >0 $ and $\tilde{\vc{c}}_{A} >0$, there exist constants $m, M >0$ such that 
\begin{align*}
m \vc{c}_{A} \leq \tilde{\vc{c}}_{A}\leq M \vc{c}_{A}.
\end{align*}
Using the monotonicity of $f(\theta)$,
\begin{align*}
\lim_{\alpha \uparrow 0} f(\alpha \vc{c}_{A}) = \lim_{\alpha \uparrow 0} f(\alpha M \vc{c}_{A}) \leq  \lim_{\alpha \uparrow 0} f(\alpha\tilde{\vc{c}}_{A}) \leq \lim_{\alpha \uparrow 0} f(\alpha m \vc{c}_{A}) = \lim_{\alpha \uparrow 0} f(\alpha \vc{c}_{A}).
\end{align*}
\end{proof}
In view of this lemma, we write 
\begin{align*}
f({0}_{A}-) \equiv \lim_{\alpha \uparrow 0} f(\alpha\vc{c}_{A}).
\end{align*} 
The following proposition is one of the main tools used to prove Theorem \ref{thr:heavy traffic 1}. It will allow us to show that whenever the sequence of the MGFs of $Z^{(n)}(\infty)$ has a limit, it must be left continuous at 0.
\begin{proposition}
\label{lem:int_geq_boundary}
Assume $\psi\in {\cal C}$ and $\{\psi_j, j \in \J\} \subset {\cal C}$ satisfy
\begin{align}
  \gamma(\theta) \psi(\theta) + \sum_{j \in \J} b_j\gamma_j(\theta)  \psi_{j}(\theta) = 0, \quad \text{ for }{\theta} \le {0}, \label{eq:limitBAR}
\end{align}
\blue{and that $\psi_j(\theta)$ is independent of $\theta_j$.} For any $A \subset \J$,
\begin{align} \label{eq:A 3}
\psi({0}_{A}-) \geq \psi_{j}({0}_{A}-), \quad  j \in \J.
\end{align}
Furthermore, when $A = \J$ we have
\begin{align} \label{eq:A 3.1}
\psi({0}_{\J}-) = \psi_{j}({0}_{\J}-), \quad j \in \J.
\end{align}
\end{proposition}

\begin{proof}
Recall the definitions of $\gamma(\theta)$ and $\gamma_j(\theta)$ from \eqref{eq:gamma 1}. For any $A \subset \J$, $\vc{c}_{A} > {0}$, and $\alpha < 0$, we set ${\theta} = \alpha \vc{c}_{A}$ in \eqref{eq:limitBAR} to get
\begin{align*}
  \Big( \frac 12 \alpha^{2} \br{\vc{c}_{A}, \Sigma \vc{c}_{A}} - \alpha \sum_{j \in \J} b_j \br{{\vc{c}_{A}}, R^{(j)}} \Big) \psi(\alpha \vc{c}_{A}) + \alpha \sum_{j \in \J}  b_{j}  \br{\vc{c}_{A}, R^{(j)}} \psi_{j}(\alpha \vc{c}_{A}) = 0. 
\end{align*}
We divide both sides by $\alpha$ and let $\alpha  \uparrow 0$, which yields
\begin{align*}
  - \sum_{j \in \J} \br{\vc{c}_{A}, R^{(j)}} b_{j} \psi({0}_{A}-) + \sum_{j \in \J} \br{\vc{c}_{A}, R^{(j)}} b_{j} \psi_{j}({0}_{A}-) = 0. 
\end{align*}
By \lem{ray independent},  $\vc{c}_{A} = (c_1, ... , c_d)^{\rs{T}}$ can be arbitrary as long as $\vc{c}_{A} > {0}$. For each fixed $i \in A$, we set $c_i = 1$ and let $c_k \downarrow 0$ for $k \in A \setminus \{i\}$. We arrive at
\begin{align*}
  \sum_{j \in \J} r_{ij} b_{j} & \big( \psi({0}_{A}-) - \psi_{j}({0}_{A}-) \big) = 0, \quad i \in A,
\end{align*}
where $r_{ij}$ is the $(i,j)$th entry of the reflection matrix $R$. Next, we split the summation in this formula into two parts:
\begin{align}
\label{eq:BAR 4}
 & \sum_{j \in A} r_{ij} b_{j}  \big( \psi({0}_{A}-) - \psi_{j}({0}_{A}-) \big) \\
  & \qquad + \sum_{j \in \J \setminus A} r_{ij} b_{j} \big( \psi({0}_{A}-) - \psi_{j}({0}_{A}-) \big) = 0, \quad i \in A. \nonumber 
\end{align}
We consider these equations as an $|A|$-dimensional vector equation. Let $R^{A}$ be the principal sub-matrix of $R$ whose entry indices are taken from $A$. Since $R$ is an $\sr{M}$-matrix, $R^{A}$ is also an $\sr{M}$-matrix, so it has an inverse whose entries are all nonnegative. Denote the $(i,j)$th entry of the inverse of $R^{A}$ by $\widetilde{r}^{A}_{ij}$. Multiplying the vector version of \eq{BAR 4} from the left by $(R^{A})^{-1}$, we have
\begin{align}
\label{eq:BAR 5}
  & b_{k}  \big( \psi({0}_{A}-) - \psi_{k}({0}_{A}-) \big) \\
  & \qquad + \sum_{i \in A} \sum_{j \in \J \setminus A} \widetilde{r}^{A}_{ki} r_{ij} b_{j} \big( \psi({0}_{A}-) - \psi_{j}({0}_{A}-) \big) = 0, \qquad k \in A. \nonumber 
\end{align}
Set $A = \J \setminus B$, where $B \subset \J$ and $B \ne \J$. Using induction on the size of the set $B$, we will show that 
\begin{align*}
  \psi({0}_{\J \setminus B}-) - \psi_{k}({0}_{\J \setminus B}-) \ge 0, \quad k \in \J.
\end{align*}
We first take $B = \emptyset$, meaning $A=\J$. Since the summation in \eq{BAR 5} vanishes, we have
\begin{align}
\label{eq:A 1}
  \psi({0}-) = \psi_{k}({0}-), \quad k \in \J,
\end{align}
which proves \eqref{eq:A 3.1}. Hence, the base case is true and we now justify the inductive step.
For a general set $B \subset \J$, \eq{BAR 5} becomes
\begin{align}
\label{eq:BAR 6}
  & b_{k} \big( \psi({0}_{\J \setminus B}-) - \psi_{k}({0}_{\J \setminus B}-) \big) \\
  &\qquad + \sum_{i \in \J \setminus B} \sum_{j \in B} \widetilde{r}^{\J \setminus B}_{ki} r_{ij} b_{j} \big( \psi({0}_{\J \setminus B}-) - \psi_{j}({0}_{\J \setminus B}-) \big) = 0 
\nonumber 
\end{align}
for each $k \in \J \setminus B$.
For $j \in B$ we have
\begin{align*}
\ \psi({0}_{\J \setminus B}-) - \psi_{j}({0}_{\J \setminus B}-) =&\  \psi({0}_{\J \setminus B}-) - \psi_{j}({0}_{\J \setminus (B \setminus \{j\})}-) \\
   \geq &\ \psi({0}_{\J \setminus (B \setminus \{j\})}-) - \psi_{j}({0}_{\J \setminus (B \setminus \{j\})}-),
\end{align*}
where 
\begin{align*}
\psi({0}_{\J \setminus B}-) \geq \psi({0}_{\J \setminus (B \setminus \{j\})}-)
\end{align*}
follows \blue{because $\psi\in {\cal C}$}. We now assume that 
\begin{align*}
 \psi({0}_{\J \setminus (B \setminus \{j\})}-) - \psi_{k}({0}_{\J \setminus (B \setminus \{j\})}-)  \geq 0, \quad k \in \J
\end{align*} 
to immediately obtain 
\begin{align} \label{eq:geq_first_part}
\psi({0}_{\J \setminus B}-) - \psi_{j}({0}_{\J \setminus B}-) \geq 0, \quad j \in B.
\end{align}
Note that $r_{ij} \le 0$ for $j \neq i$ and $\widetilde{r}^{\J \setminus B}_{ki} \ge 0$ because  $R$ and $R^{A}$ are $\sr{M}$-matrices. Hence, 
\begin{align*}
\widetilde{r}^{\J \setminus B}_{ki} r_{ij} b_{j} \big( \psi({0}_{\J \setminus B}-) - \psi_{j}({0}_{\J \setminus B}-) \big) \leq 0, \quad j \in B,
\end{align*}
which together with \eqref{eq:BAR 6} immediately gives us
\begin{align} \label{eq:geq_second_part}
\psi({0}_{\J \setminus B}-) - \psi_{k}({0}_{\J \setminus B}-) \geq 0, \quad k \in \J \setminus B.
\end{align}
We combine \eqref{eq:geq_first_part} and \eqref{eq:geq_second_part} to complete the proof.
\end{proof}

\section{Proof of \thr{heavy traffic 1}}
\label{sect:proofs}
\setnewcounter

For $n \geq 1$, recall that $\varphi^{(n)}({\theta})$ and $\varphi^{(n)}_{j}({\theta})$ are MGFs defined in \eqref{eq:laplaceprelimit}. 
To prove Theorem \ref{thr:heavy traffic 1}, it suffices to prove
\begin{align}
  \label{eq:mgfconv}
  & \lim_{n\to\infty} \varphi^{(n)}({\theta}) = \varphi(\theta) \\
  &  \lim_{n\to\infty} \varphi^{(n)}_j({\theta}) = \varphi_j(\theta), \quad j\in  \J, \qquad \text{ for }  \theta \le 0, \notag
\end{align}
where $\varphi$ and $\varphi_j$ are MGFs defined in (\ref{eq:mgfsrbmdef}).

We now state a lemma that will be used frequently in this section.
Its proof is given in  \app{laplacetightness}.

\begin{lemma}
\label{LEM:LAPLACETIGHTNESS}  
Let $\{\nu^{(n)}, n \geq 1 \}$ be a sequence of probability measures on $\R_+^d$ with corresponding MGFs 
\begin{align*}
f^{(n)}(\theta) = \int_{\R_+^{d}} e^{\br{\theta, x}} \nu^{(n)}(dx), \quad \text{ for $\theta \in \R^d$ with $\theta \leq 0$}.
\end{align*}
Suppose that there exists a function $f(\theta)$ such that
\begin{align*}
f^{(n)}(\theta) \to f(\theta) \quad \text{ pointwise for all $\theta \in \R^d$ with $\theta \leq 0$}.
\end{align*}

\begin{enumerate}[(a)]
\item \label{eq:iclaimtightness} $f(0_\J-) = 1$ if and only if $\{\nu^{(n)}, n \geq 1 \}$ is tight.
\item \label{eq:iiclaimtightness} If $f(0_\J-) = 1$, then $f(\theta)$ is the  MGF of some probability measure $\nu$ on $\R_+^d$, which immediately implies that 
\begin{align*}
\nu^{(n)} \Rightarrow \nu, \quad \text{ as $n \to \infty$}.
\end{align*}  
\end{enumerate}
\end{lemma}

\begin{proof}[Proof of Theorem~\ref{thr:heavy traffic 1}]

\dgreen{It follows from the proof of Helly's selection principle \cite[Theorem 8.1]{Gut2005}} that for every subsequence $\{n'\} \subset \{n\}_{n=1}^{\infty}$, there exists a further subsequence $\{n''\}$ and some non-decreasing functions $\psi(\theta)$ and $\psi_{j}(\theta)$, such that for every $\theta \le 0$, 
\begin{align}
\varphi^{(n'')}(\theta) \to \psi(\theta) \quad \text{ and } \quad \varphi_{j}^{(n'')}(\theta) \to \psi_{j}(\theta), \quad j \in \J. \label{eq:laplacelimit}
\end{align}
Both $\psi(\theta)$ and $\psi_{j}(\theta)$ are functions in
${\cal C}$, which was introduced in \sectn{tightness}. In particular we note that $\psi(0) = \psi_j(0)=1$. Furthermore, Proposition~\ref{lem:prelimitMGFBAR} implies that $\psi(\theta)$ and $\psi_{j}(\theta)$ satisfy BAR (\ref{eq:limitBAR}). Suppose we know that the limiting functions $\psi$ and $\psi_j$ satisfy
\begin{align}
& \psi(0_\J-) = 1, \label{eq:mgfcon}\\
& \psi_{j}(0_\J-) = 1, \quad j \in \J. \label{eq:tightnesscond}
\end{align}
Then Lemma~\ref{LEM:LAPLACETIGHTNESS}, BAR \eqref{eq:limitBAR} and the MGF version
of Lemma~\ref{lem:barneccsuff} imply that
$\psi(\theta)=\varphi(\theta)$ and
$\psi_{j}(\theta)=\varphi_j(\theta)$. Since $\{n'\}$ was an arbitrary
subsequence, we have proved (\ref{eq:mgfconv}), and consequently Theorem \ref{thr:heavy traffic 1}.  

In the remainder of the proof, we prove (\ref{eq:mgfcon}), which implies (\ref{eq:tightnesscond}) by \eqref{eq:A 3.1}.  Let $A = \{j\}$, then \eqref{eq:A 3} implies
\begin{align}
\label{eq:A 4}
  \psi({0}_{\{j\}}-) - \psi_{j}({0}_{\{j\}}-) \ge 0, \quad j \in \J.
\end{align}
Since $\psi_{j}({0}_{\{j\}}-) = \psi_{j}(0) = 1$, we have
\begin{align}
\psi({0}_{\{j\}}-) = 1, \quad j \in \J. \label{eq:firstseed}
\end{align}
Recall that 
\begin{align*}
Z^{(n'')}(\infty) =  ({Z}^{(n'')}_{1}(\infty), \ldots , {Z}^{(n'')}_{d}(\infty))
\end{align*}
is the scaled steady state queue length vector of the $n''$th GJN. Applying Lemma~\ref{LEM:LAPLACETIGHTNESS} to \eqref{eq:firstseed}, we see that $\{{Z}^{(n'')}_{j}(\infty)\}$ is tight for each $j \in \J$. Since all the marginals are tight, $\{{Z}^{(n'')}(\infty)\}$ is tight as well. We invoke Lemma~\ref{LEM:LAPLACETIGHTNESS} again to conclude that  (\ref{eq:mgfcon}) holds.
\end{proof}

\section{Concluding remarks}
\label{app:concluding}
\setnewcounter
To summarize, this paper uses a novel approach to justify the steady-state diffusion approximation of GJNs. Our method does not rely on first using the process limit $Z^{(n)}(\cdot) \Rightarrow Z(\cdot)$, followed by the tightness of $\{Z^{(n)}(\infty),\ n \geq 1\}$ as in \cite{GamaZeev2006} or \cite{BudhLee2009}. Rather, we work directly with the basic adjoint relationships of the GJN and the SRBM, which characterize their respective stationary distributions. By applying the BAR \eqref{eq:hard_BAR} of the GJN to a carefully designed exponential test function, we show in Proposition~\ref{lem:prelimitMGFBAR} that the sequence of moment generating functions of the GJN (together with a corresponding sequence of boundary measures) asymptotically satisfies the BAR \eqref{eq:srbmbar} of the SRBM. 

In addition to Proposition~\ref{lem:prelimitMGFBAR}, we also require tightness of the sequence $\{Z^{(n)}(\infty),\ n \geq 1\}$ like in \cite{GamaZeev2006} and \cite{BudhLee2009}. However, unlike the previous two papers, we do not rely on constructing Lyapunov functions to prove this tightness. Our tightness argument from Section~\ref{sect:proofs}, relies on algebraic manipulations of the BAR \eqref{eq:srbmbar} of the SRBM, the bulk of which are presented as Proposition~\ref{lem:int_geq_boundary}. A critical condition for this proposition is that the reflection matrix of the SRBM is an $\cal M$-matrix\blue{, and in particular, that its diagonal entries are non-negative and off diagonal entries are non-positive. }

An important direction for future research is generalizing the methodology presented in this paper to the multiclass queueing network setting, where new sources of difficulty emerge. One source of difficulty is the need to handle state space collapse. In a multiclass network with $d$ stations, the vector of queue lengths is of a higher dimension than $d$, while the approximating diffusion process remains $d$-dimensional. This will manifest itself as an extra error term on the right hand side of \eqref{eq:prelimBAR} in Proposition~\ref{lem:prelimitMGFBAR}, and we would need to have some way to deal with it. Another source of difficulty is to relax the $\cal M$-matrix currently needed to prove tightness, as the ${\cal M}$-matrix structure is a feature unique to GJNs. \blue{Preliminary experiments with several multiclass networks suggest that a variant of Proposition~\ref{lem:int_geq_boundary} can be established even when the reflection matrix is not an ${\cal M}$-matrix. However, it remains an open problem to find the minimal conditions on the reflection matrix under which our approach can be carried out.}

\subsection*{Acknowledgements}
This research is supported in part by NSF Grants CNS-1248117, CMMI-1335724, and
CMMI-1537795, and by JSPS KAKENHI Grant Number 16H027860001.

\bigskip

\appendix

\noindent{\bf \Large Appendix}

\section{Proofs of Lemmas in Section~\ref{SECT:PRELIMIT}}
\label{app:full_gen_proofs}

In this section, we prove the lemmas that were stated but not proved in Section~\ref{SECT:PRELIMIT}. 

\subsection{Proof of Lemma~\ref{lem:concave 1}}
\label{app:concave 1}

The fact that $f(x)$ is well defined and finite is argued in the same way as in Lemma~\ref{lem:terminal condition 1}. We want to perform Taylor expansion on $f(x)$ and to do so  we first show that it is infinitely differentiable. Observe that the function 
\begin{align*}
  {G}(y) = \E (e^{y H}), \qquad y \in \R,
\end{align*}
is well defined and belongs to $C^{\infty}(\R)$. Hence, $F(x,y) \equiv {G}(y) - e^{-x}$ belongs to $C^{\infty}(\R^2)$. Since
\begin{align*}
  \frac {\partial}{\partial y} F(x,y) = \dd{E}(H e^{yH}) \not= 0, \qquad (x,y) \in \R^2,
\end{align*}
the implicit function theorem \cite[Theorem 3.3.1]{KranPark2013} implies that there is a single valued and infinitely differentiable function $y(x)$ of $x$ which is the solution of $F(x,y(x))=0$ for each \blue{$x \in \big(-\infty, -\log(\Prob(H = 0)) \big)$}. Therefore $f(x) = y(x)$ for $x \in \big(-\infty, -\log(\Prob(H = 0)) \big)$, proving (\ref{item:concave1_b}). The first and second derivatives of $f(x)$ are 
\begin{align}\label{eq:eta_D2}
& f'(x) =  -\frac{\E(e^{f(x) H} )}{\E( H e^{f(x) H} )}, \\
& f''(x) = - f'(x)\Bigg(1- \frac{\E(e^{f(x) H}) \E( H^2 e^{f(x) H} )}{\big(\E( H e^{f(x) H} )\big)^2}\Bigg) \notag .
\end{align}
Observe that $f'(x) < 0$, which implies that $f(x)$ is decreasing. To prove concavity, we wish to show that $f''(x) \leq 0$. This follows from 
\begin{align*}
\big(E( H e^{f(x) H} )\big)^2 \leq  \E(e^{f(x) H}) \E( H^2 e^{f(x) H}),
\end{align*} 
which is just the Cauchy-Schwarz inequality. A second order Taylor expansion of $f(x)$ combined with the facts that $f(0)=0$, $f'(0) = -\lambda_H$ and $f''(0) = -\lambda_H^{3} {\sigma}_{H}^{2}$ immediately implies \eqref{eq:generic asymptotic 1}. Lastly, since $f(x)$ is concave,
\begin{align*}
f(x) \leq f'(0) x \leq \lambda_H \abs{x}, \quad x \leq 0.
\end{align*}
Furthermore, the function $\abs{f(x)} = -f(x)$ is convex when restricted to $x \geq 0$. Therefore, 
\begin{align*}
\abs{f(x)} \leq \frac{\abs{f(K)} - \abs{f(0)}}{K-0} x = \frac{\abs{f(K)}}{K} x, \quad 0 \leq x \leq K.
\end{align*}

\subsection{Proof of Lemma~\ref{lem:uniform 1}}
\label{app:uniform1_gen_proof}
We begin by proving \eqref{eq:uniform_eta}, or
\begin{align}\label{eq:generalization_needed_lemma_uniform1}
& \lim_{n \to \infty} \sup_{0 < |\theta_i|<  K} \frac{c^{(n)}_{e,i}(r_n \theta_i)}{r_n^2 \theta_i^2}  \\
& \quad = \lim_{n \to \infty} \sup \limits_{\abs{y} \leq K} \frac{1}{2} \big| \big(\eta_{i}^{(n)}\big)''(r_n y) - \big(\eta_{i}^{(n)}\big)''(0) \big| = 0, \qquad i \in \sr{E}. \notag
\end{align} 
 Using \eqref{eq:eta_D2}, we see that $\big(\eta_{i}^{(n)}\big)''(\theta_i)$ can be written symbolically as 
\begin{align*}
\big(\eta_{i}^{(n)}\big)''(\theta_i) = \frac{a_1(\theta_i)}{a_2(\theta_i)} \Bigg( 1 - \frac{a_1(\theta_i) a_3(\theta_i)}{(a_2(\theta_i))^2} \Bigg),
\end{align*}
where 
\begin{align*}
& a_1(\theta_i) = \E \big(e^{\eta_i^{(n)}(\theta_i) g^{(n)}(T_{e,i}^{(n)})} \big)= e^{-\theta_i}, \\
& a_2^{(n)}(\theta_i) = \E \big( g^{(n)}(T_{e,i}^{(n)}) e^{\eta_i^{(n)}(\theta_i) g^{(n)}(T_{e,i}^{(n)})} \big), \\ 
& a_3^{(n)}(\theta_i) = \E \big( (g^{(n)}(T_{e,i}^{(n)}))^2 e^{\eta_i^{(n)}(\theta_i) g^{(n)}(T_{e,i}^{(n)})} \big).
\end{align*}
To prove \eqref{eq:generalization_needed_lemma_uniform1}, we want to replace each $a_i(r_n y)$ by $a_i(0)$ one at a time. For this, we need to show that
\begin{align*}
&\sup_{\substack{n \geq 1 \\ \abs{y} \leq K}} a_1(r_n y) < \infty, \quad \sup_{\substack{n \geq 1 \\ \abs{y} \leq K}} a_2^{(n)}(r_n y) < \infty,\\
&\sup_{\substack{n \geq 1 \\ \abs{y} \leq K}} \frac{1}{a_2^{(n)}(r_n y)} < \infty,  \quad \sup_{\substack{n \geq 1 \\ \abs{y} \leq K}}a_3^{(n)}(r_n y) < \infty,
\end{align*}
and
\begin{align}
&\lim_{n \to \infty} \sup \limits_{\abs{y} \leq K} \big| a_1(r_n y) - a_1(0)\big| = 0, \notag \\
&\lim_{n \to \infty} \sup \limits_{\abs{y} \leq K} \Big|a_2^{(n)}(r_n y) - a_2^{(n)}(0)\Big| = 0, \label{eq:ingred1} \\
&\lim_{n \to \infty} \sup \limits_{\abs{y} \leq K} \big| a_3^{(n)}(r_n y) - a_3^{(n)}(0)\big| = 0. \label{eq:ingred2}
\end{align}
We begin with $a_1(\theta_i) = e^{-\theta_i}$:
\begin{align*}
\sup_{\substack{n \geq 1 \\ \abs{y} \leq K}} e^{r_ny} < \infty \quad \text{ and } \quad \lim_{n \to \infty} \sup \limits_{\abs{y} \leq K} \big| e^{-r_n y} - 1\big| = 0.
\end{align*}
Moving on to $a_2^{(n)}(\theta_i)=\E \big( g^{(n)}(T_{e,i}^{(n)}) e^{\eta_i^{(n)}(\theta_i) g^{(n)}(T_{e,i}^{(n)})} \big)$; we first prove \eqref{eq:ingred1}.  By \eqref{eq:eta asymptotic 0}, we have
\begin{align} \label{eq:etarn_bound}
\abs{\eta_i^{(n)}(r_n y) g^{(n)}(T^{(n)}_{e,i})} \leq \hat{c}_{e,i}^{(n)}(r_nK) \abs{y}, \quad \abs{y} \leq K.
\end{align} 
Now observe that
\begin{align*}
& \lim_{n\to\infty} \sup \limits_{\abs{y} \leq K} \Bigg| \E g^{(n)}(T^{(n)}_{e,i}) - \E g^{(n)}(T^{(n)}_{e,i}) e^{ \eta_i^{(n)}(r_n y) g^{(n)}(T^{(n)}_{e,i}) } \Bigg|\notag \\ 
&\quad \leq \lim_{n\to\infty} \sup \limits_{\abs{y} \leq K} \Bigg| \E \big(g^{(n)}(T^{(n)}_{e,i})\big)^2\eta_i^{(n)}(r_n y) e^{\hat{c}_{e,i}^{(n)}(r_nK)\abs{y}} \Bigg|\\
& \quad \leq\ \lim_{n\to\infty}  r_n K\hat{c}_{e,i}^{(n)}(r_nK) e^{\hat{c}_{e,i}^{(n)}(r_nK)K}\E \big(T^{(n)}_{e,i}\big)^2 = 0.
\end{align*}
In the first inequality, we use \eqref{eq:etarn_bound} together with the bound
\begin{align} \label{eq:exp_ineq}
\abs{e^{x}-1} \leq \abs{x}e^{\abs{x}}, \quad x \in \R.
\end{align}
The second inequality is by \eqref{eq:eta asymptotic 0} and the fact that the limit equals zero follows from \eqref{eq:uniarr} and Lemma~\ref{lem:uniform_slopes}, proving \eqref{eq:ingred1}. By \eqref{eq:momentarr}, we know that 
\begin{align*}
\lim_{n \to \infty} a_2^{(n)}(0) = \frac{1}{\lambda_{e,i}} > 0.
\end{align*} 
Combined with \eqref{eq:ingred1}, this implies 
\begin{align*}
\sup_{\substack{n \geq 1 \\ \abs{y} \leq K}} a_2^{(n)}(r_n y) < \infty \quad \text{ and } \quad \sup_{\substack{n \geq 1 \\ \abs{y} \leq K}} \frac{1}{a_2^{(n)}(r_n y)} < \infty.
\end{align*}


Lastly, we tackle $a_3^{(n)}(\theta_i) = \E \big( (g^{(n)}(T_{e,i}^{(n)}))^2 e^{\eta_i^{(n)}(\theta_i) g^{(n)}(T_{e,i}^{(n)})} \big)$, first proving \eqref{eq:ingred2}. We have
\begin{align*}
& \sup \limits_{\abs{y} \leq K} \Big| \E \big( g^{(n)}(T^{(n)}_{e,i}) \big)^2 e^{\eta_i^{(n)}(r_n y) g^{(n)}(T^{(n)}_{e,i})} - \E \big( g^{(n)}(T^{(n)}_{e,i}) \big)^2 \Big| \notag \\
&\quad \leq  \sup \limits_{\abs{y} \leq K}  \E \Big|\big( g^{(n)}(T^{(n)}_{e,i}) \big)^3 \eta_i^{(n)}(r_n y)e^{\hat{c}_{e,i}^{(n)}(r_nK)K}\Big| \notag \\
&\quad \leq r_nK \hat{c}_{e,i}^{(n)}(r_nK) e^{\hat{c}_{e,i}^{(n)}(r_nK)K} \E \big(T^{(n)}_{e,i} \wedge 1/r_n \big)^3 \notag \\
&\quad \leq  K\hat{c}_{e,i}^{(n)}(r_nK) e^{\hat{c}_{e,i}^{(n)}(r_nK)K} \\
& \qquad\qquad \times \E \big[ \big(T^{(n)}_{e,i}\big)^2 \big(r_nT^{(n)}_{e,i} \wedge 1 \big)\big( 1(T^{(n)}_{e,i} > r_n^{-1/2}) + 1(T^{(n)}_{e,i} \leq r_n^{-1/2}) \big)\big] \notag \\
&\quad \leq  K\hat{c}_{e,i}^{(n)}(r_nK) e^{\hat{c}_{e,i}^{(n)}(r_nK)K} \E \big[\big(T^{(n)}_{e,i}\big)^2 1(T^{(n)}_{e,i} > r_n^{-1/2}) + \big(T^{(n)}_{e,i}\big)^2 r_n^{1/2}\big] \to 0 
\end{align*}
 as $n \to \infty$.
The first inequality is obtained by \eqref{eq:etarn_bound} and \eqref{eq:exp_ineq} and the second one by \eqref{eq:eta asymptotic 0}. Convergence to zero is justified by \eqref{eq:uniarr} and Lemma~\ref{lem:uniform_slopes} (whose proof in Section~\ref{app:uniform_slopes} does not use Lemma~\ref{lem:uniform 1}), which proves \eqref{eq:ingred2}. Finally, \eqref{eq:momentarr} implies 
\begin{align*}
\lim_{n \to \infty} a_3^{(n)}(0) < \infty.
\end{align*}
Combined with \eqref{eq:ingred2}, this gives us
\begin{align*}
\sup_{\substack{n \geq 1 \\ \abs{y} \leq K}}a_3^{(n)}(r_n y) < \infty,
\end{align*}
which concludes the proof of \eqref{eq:generalization_needed_lemma_uniform1}.

The second claim of this lemma, \eqref{eq:uniform_zeta}, is now simple to verify. Indeed, recalling the form of $c^{(n)}_{s,j}(\theta)$ from \eqref{eq:err_zeta} and the definition of $c_{Lip,j}(K)$ from \eqref{eq:clip_def}, we have
\begin{align*}
 &\lim_{n \to \infty} \sup_{0 < \|\theta\|<  K} \frac{c^{(n)}_{s,j}(r_n\theta)}{r_n^2 \|\theta\|^2}\\
 &\quad= \lim_{n \to \infty}\sup_{\abs{y} \leq c_{Lip,j}(r_nK) r_nK} \frac{\log^2 t_j(r_n\theta)}{r_n^2 \|\theta\|^2}  \Big| \big(\chi^{(n)}_j\big)''(y) - \big(\chi^{(n)}_j\big)''(0)\Big| \\
 &\qquad+ \lim_{n \to \infty} \sup_{0 < \|\theta\|<  K} \widetilde{\lambda}^{(n)}_{s,j}\frac{c_{1,j}(r_n\theta)}{r_n^2 \|\theta\|^2}  +  \lim_{n \to \infty} \sup_{0 < \|\theta\|<  K}\frac 12 \big(\widetilde{\lambda}^{(n)}_{s,j}\big)^{3} \big(\widetilde{\sigma}^{(n)}_{s,j}\big)^{2}\frac{c_{2,j}(r_n\theta)}{r_n^2 \|\theta\|^2}.
\end{align*}
From \eqref{eq:tildes_converge}, \eqref{eq:conetwo} and the fact that $\log t_j (\theta) \in C^{\infty}(\R^d)$ its easy to see that
\begin{align*}
\lim_{n \to \infty}\sup_{0 < \|\theta\| \leq K} \frac{c_{2,j}(r_n\theta)}{r_n^2 \|\theta\|^2} = 0 \quad \text{ and } \quad \lim_{n \to \infty}\sup_{0 < \|\theta\| \leq K} \frac{c_{2,j}(r_n\theta)}{r_n^2 \|\theta\|^2} = 0.
\end{align*}
Furthermore, 
\begin{align*}
&\lim_{n \to \infty}\sup_{\abs{y} \leq c_{Lip,j}(r_nK) r_nK} \frac{\log^2 t_j(r_n\theta)}{r_n^2 \|\theta\|^2}  \Big| \big(\chi^{(n)}_j\big)''(y) - \big(\chi^{(n)}_j\big)''(0)\Big|\\
&\quad\leq \lim_{n \to \infty}\sup_{\abs{y} \leq c_{Lip,j}(r_nK) r_nK} c_{Lip,j}^{2}(r_nK)  \Big| \big(\chi^{(n)}_j\big)''(y) - \big(\chi^{(n)}_j\big)''(0)\Big| = 0,
\end{align*}
because by the definition of $\chi_j^{(n)}$ in \eqref{eq:chi 1}, it is not hard to repeat the arguments used in this proof to see that a version \eqref{eq:generalization_needed_lemma_uniform1} holds for $\chi_j^{(n)}$ as well.

\subsection{Proof of Lemma~\ref{lem:uniform_slopes}}
\label{app:uniform_slopes}
We want to show that for all $i \in \sr{E}$ and all $K \in (0,r_nM]$,
\begin{align}
\sup_{n \geq 1} \hat{c}_{e,i}^{(n)}(r_nK) = \sup_{n \geq 1} \Big\{ \max \big{\{}\widetilde{\lambda}^{(n)}_{e,i}, \abs{\eta^{(n)}_i(r_nK)}/r_nK \big{\}} \Big\} < \infty. \label{eq:slopes_wts}
\end{align}
 By \eqref{eq:tildes_converge}, $\widetilde{\lambda}_{e,i}^{(n)} \to \lambda_{e,i} < \infty$. Now since $\eta^{(n)}_i(\theta_i)$ is finite for each fixed $n$, it is enough to show that 
\begin{align*}
\limsup_{n \to \infty} \frac{\abs{\eta^{(n)}_i(r_nK)}}{r_n K } < \infty.
\end{align*} 
Recalling the definition of $\eta^{(n)}_i(\theta_i)$ in \eqref{eq:eta 1} and the truncation function $g^{(n)}(y) = \min(y, 1/r_n)$ defined for $y \in \R$, we observe that
\begin{align}\label{eq:deriv1bound}
 \frac{\abs{\eta^{(n)}_i(r_nK)}}{r_n K } & \leq \sup_{0 \leq y \leq r_n K}\big| \big(\eta^{(n)}_i\big)'(y)\big| \\
 & = \sup_{0 \leq y \leq r_n K} \frac{e^{-y}}{\E\big(g^{(n)}(T_{e,i}^{(n)})e^{g^{(n)}(T_{e,i}^{(n)}) \eta_i^{(n)}(y)}\big)} \notag \\
 & \leq \frac{1}{\E\big(g^{(n)}(T_{e,i}^{(n)})e^{g^{(n)}(T_{e,i}^{(n)}) \eta_i^{(n)}(r_n K)}\big)}, \notag
\end{align}  
where in the last inequality, we used part (\ref{item:concave1_c}) of Lemma~\ref{lem:concave 1}, which states that $\eta^{(n)}_i(\theta_i)$ is a decreasing function of $\theta_i$. \blue{To bound this term, we first argue that there exists an $\epsilon \in (0,1)$ and $y_1 > 0$ such that 
\begin{align}
\liminf_{n \to \infty} \Prob(T_{e,i}^{(n)} \geq y_1) \geq \epsilon. \label{eq:problower}
\end{align}
Suppose \eqref{eq:problower} is not true. Then $\liminf_{n \to \infty} \Prob(T_{e,i}^{(n)} \geq y) < \epsilon$ for any $\epsilon \in (0,1)$ and $y > 0$. On the other hand, it follows from the uniform integrability assumption \eqref{eq:uniarr} that for any $\epsilon'$, there exists $a_0$ such that for any $a \geq a_0$, 
\begin{align*}
\E \big(T_{e,i}^{(n)} 1(T_{e,i}^{(n)} > a)\big) < \epsilon', \quad n \geq 1.
\end{align*}
Hence,
\begin{align*}
\E(T_{e,i}^{(n)}) & = \E \big(T_{e,i}^{(n)} 1(T_{e,i}^{(n)} \leq y)\big) + \E \big(T_{e,i}^{(n)} 1(T_{e,i}^{(n)} \in (y,a])\big) \\
& \qquad + \E \big(T_{e,i}^{(n)} 1(T_{e,i}^{(n)} > a)\big) \\
& \leq y + a \Prob(T_{e,i}^{(n)} > y) + \epsilon',
\end{align*}
and 
\begin{align*}
\lim_{n \to \infty} \E(T_{e,i}^{(n)}) = \liminf_{n \to \infty} \E(T_{e,i}^{(n)}) \leq y + \liminf_{n \to \infty}a \Prob(T_{e,i}^{(n)} > y) + \epsilon' \leq  y + a \epsilon + \epsilon'.
\end{align*}
Taking $\epsilon, y \to 0$ and then $\epsilon' \to 0$, we conclude that $\lim_{n \to \infty} \E(T_{e,i}^{(n)}) = 0$, which contradicts $\lambda_{e,i} < \infty$ in \eqref{eq:momentarr} and therefore proves \eqref{eq:problower}. }In addition to \eqref{eq:problower}, the uniform integrability assumption \eqref{eq:uniarr} implies that there exists a $y_2 > y_1$ such that 
\begin{align}
\sup_{n \geq 1} \Prob(T_{e,i}^{(n)} \geq y_2) < \epsilon/2. \label{eq:probupper}
\end{align}
Fixing such $y_1$ and $y_2$, and recalling that $\eta_i^{(n)}(r_n K) < 0$ for all $n \geq 1$, we see that 
\begin{align*}
&\frac{1}{\E\big(g^{(n)}(T_{e,i}^{(n)})e^{g^{(n)}(T_{e,i}^{(n)}) \eta_i^{(n)}(r_n K)}\big)} \\
& \quad \leq \frac{1}{\E\big(g^{(n)}(T_{e,i}^{(n)})e^{g^{(n)}(T_{e,i}^{(n)}) \eta_i^{(n)}(r_n K)}1(y_1 \leq T_{e,i}^{(n)} < y_2)\big)} \\
& \quad \leq \frac{1}{\E\big(g^{(n)}(y_1)e^{g^{(n)}(y_2) \eta_i^{(n)}(r_n K)}1(y_1 \leq T_{e,i}^{(n)} < y_2)\big)} \\
&\quad =\ \frac{e^{g^{(n)}(y_2) \abs{\eta_i^{(n)}(r_n K)}}}{g^{(n)}(y_1)\Prob(y_1 \leq T_{e,i}^{(n)} < y_2)}.
\end{align*}
We recall our choices of $y_1$ and $y_2$ from \eqref{eq:problower} and \eqref{eq:probupper} to arrive at 
\begin{align*}
\limsup_{n \to \infty} \frac{\abs{\eta^{(n)}_i(r_nK)}}{r_n K } \leq&\ \limsup_{n \to \infty} \frac{e^{g^{(n)}(y_2) \abs{\eta_i^{(n)}(r_n K)}}}{g^{(n)}(y_1)\Prob(y_1 \leq T_{e,i}^{(n)} < y_2)} \\
\leq&\ \limsup_{n \to \infty} \frac{2}{y_1\epsilon}e^{y_2 \abs{\eta_i^{(n)}(r_n K)}} .
\end{align*}
Since $\eta_i^{(n)}(r_nK) < 0$, it remains to show that 
\begin{align*}
\liminf_{n \to \infty}\eta^{(n)}_i(r_nK) > -\infty.
\end{align*}
Now for any $y > 0$,
\begin{align*}
e^{-r_nK} = \E \big(e^{g^{(n)}(T_{e,i}^{(n)}) \eta_i^{(n)}(r_nK)}\big) \leq&\ e^{g^{(n)}(y) \eta_i^{(n)}(r_nK)}\Prob(T_{e,i}^{(n)} > y) + \Prob(T_{e,i}^{(n)} < y)\\
 \leq&\ e^{g^{(n)}(y) \eta_i^{(n)}(r_nK)} + \Prob(T_{e,i}^{(n)} < y).
\end{align*}
We choose $y=y_1$ from \eqref{eq:problower}, and use the fact that $e^{-r_nK} \to 1$, to see that
\begin{align*}
\liminf_{n \to \infty}\Big( e^{-r_nK} - \Prob(T_{e,i}^{(n)} < y_1)\Big) = 1 - \limsup_{n \to \infty}\Prob(T_{e,i}^{(n)} < y_1) \geq 1 - (1 - \epsilon)  = \epsilon.
\end{align*}
Therefore, 
\begin{align*}
\liminf_{n \to \infty} \eta_i^{(n)}(r_nK) \geq&\ \liminf_{n \to \infty} \frac{1}{g^{(n)}(y_1)} \log \big(e^{-r_nK} - \Prob(T_{e,i}^{(n)} < y_1)\big)\\
=&\ \frac{1}{y_1} \liminf_{n \to \infty} \log \big(e^{-r_nK} - \Prob(T_{e,i}^{(n)} < y_1)\big), \\
\geq&\ \frac{1}{y_1} \log (\epsilon),
\end{align*}
where the last inequality is obtained by taking the limit infimum inside the logarithm. This proves \eqref{eq:slopes_wts}. The argument to show that
\begin{align*}
\sup_{n \geq 1} \hat{c}_{s,j}^{(n)}(r_nK) = \sup_{n \geq 1} c_{Lip,j}(r_nK)  \max \big{\{}\widetilde{\lambda}^{(n)}_{s,j}, \abs{\chi^{(n)}_j(\widehat{r_nK})}/\widehat{r_nK} \big{\}} < \infty
\end{align*}
is nearly identical to the one for $\hat{c}_{e,i}^{(n)}(r_nK)$, and requires two slight modifications. First, observe that
\begin{align*}
\sup_{n \geq 1} c_{Lip,j}(r_nK) < \infty,
\end{align*}
which is a trivial consequence of its definition in \eqref{eq:clip_def}. 
Second, the equality in \eqref{eq:deriv1bound} would be modified to 
\begin{align*}
\sup_{\substack{\|y \| \leq r_n K \\y \in \R^d}}\big| \big(\chi^{(n)}_j\big)'(\log t_j(y))\big| =&\ \sup_{\substack{\|y \| \leq r_n K \\y \in \R^d}} \frac{e^{-\log t_j(y)}}{\E\big(g^{(n)}(T_{s,j}^{(n)})e^{g^{(n)}(T_{s,j}^{(n)}) \chi_j^{(n)}(\log t_j(y))}\big)} \\
\leq&\ \sup_{\substack{\|y \| \leq r_n K \\y \in \R^d}} \frac{e^C}{\E\big(g^{(n)}(T_{s,j}^{(n)})e^{g^{(n)}(T_{s,j}^{(n)}) \chi_j^{(n)}(\log t_j(y))}\big)},
\end{align*}
where $C > 0$ is an upper bound for $\sup_{\substack{\|y \| \leq r_n K}} \abs{\log t_j(y)}$ that is  independent of $n$, and $t_j(\theta)$ is defined in \eqref{eq:zeta 1}.

\subsection{Proofs of Lemmas~\ref{lem:boundary_probs} and \ref{lem:residualUI} }
\label{app:bp+ru}
Let $\nu^{(n)}$ be the probability measure corresponding to $X^{(n)}(\infty)$. For $i \in \sr{E}$ and $j \in \J$, recall the external arrival process $\{E_i^{(n)}(t), t \geq 0\}$ and the departure process $\{D_j^{(n)}(t), t \geq 0\}$ defined in \eqref{eq:arrival_process} and \eqref{eq:departure_process}, respectively. Set 
\begin{align*}
E_i^{(n)}(0,t] = E_i^{(n)}(t) - E_i^{(n)}(0) \quad \text{ and } \quad D_j^{(n)}(0,t] = D_j^{(n)}(t) - D_j^{(n)}(0).
\end{align*} 
In this section we prove the following five statements, which imply Lemmas~\ref{lem:boundary_probs} and \ref{lem:residualUI}. We show that for all $i \in \sr{E}$ and $j \in \J$, any real number $K > 0$ and any integer $n \geq 1$,  
\begin{align}
&\Prob(L_j^{(n)}(\infty) > 0) = \frac{\E_{\nu^{(n)}}\big[ D_j^{(n)}(0,1]\big]}{\lambda_{s,j}^{(n)}},\label{eq:bar_item1}\\
&\E_{\nu^{(n)}} \big[ E_i^{(n)}(0,t]\big] =  \lambda_{e,i}^{(n)} t, \label{eq:bar_item2} \\
&\E_{\nu^{(n)}} \big[ D_j^{(n)}(0,t]\big] = \lambda_{a,j}^{(n)} t,\label{eq:bar_item3} \\
&\E \big[R_{s,j}^{(n)}(\infty) 1(R_{s,j}^{(n)}(\infty) > K)1(L_j^{(n)}(\infty)>0)  \big] \label{eq:bar_item4} \\
&\quad = \frac{\E_{\nu^{(n)}}\big[ D_j^{(n)}(0,1]\big]}{2}\E \big[(T_{s,j}^{(n)})^2 1(T_{s,j}^{(n)} > K) - K^2 1(T_{s,j}^{(n)} > K) \big], \notag\\
& \E \big[R_{e,i}^{(n)}(\infty) 1(R_{e,i}^{(n)}(\infty) > K)  \big] \label{eq:bar_item5}\\
&\quad = \frac{\E_{\nu^{(n)}}\big[ E_i^{(n)}(0,1]\big]}{2}\E \big[(T_{e,i}^{(n)})^2 1(T_{e,i}^{(n)} > K) - K^2 1(T_{e,i}^{(n)} > K) \big]. \notag
\end{align}
Assuming these five equations are correct, we see that Lemma~\ref{lem:boundary_probs} is immediately implied by \eqref{eq:bar_item1} and \eqref{eq:bar_item3}. To see why Lemma~\ref{lem:residualUI} is true, we write 
\begin{align*}
&\ \E \big[R_{s,j}^{(n)}(\infty) 1(R_{s,j}^{(n)}(\infty) > K)  \big]\\
&\quad = \ \E \big[R_{s,j}^{(n)}(\infty) 1(R_{s,j}^{(n)}(\infty) > K)1(L_j^{(n)}(\infty)>0)  \big]
 \\
 & \qquad + \E \big[R_{s,j}^{(n)}(\infty) 1(R_{s,j}^{(n)}(\infty) > K)1(L_j^{(n)}(\infty)=0)  \big] \\
& \quad =\ \E \big[R_{s,j}^{(n)}(\infty) 1(R_{s,j}^{(n)}(\infty) > K)1(L_j^{(n)}(\infty)>0)  \big]
 \\
 & \qquad + \E \big[T_{s,j}^{(n)}1(T_{s,j}^{(n)} > K) \big] \Prob(L_j^{(n)}(\infty)=0),
\end{align*}
where in the last equality we use  $\big[ R_{s,j}^{(n)}(\infty) |L_j^{(n)}(\infty) =0\big] \stackrel{d}{=} T_{s,j}^{(n)}$, which is true by construction. Uniform integrability of $\{R_{e,i}^{(n)}(\infty), n \geq 1\}$ and $\{R_{s,j}^{(n)}(\infty), n \geq 1\}$ then follows from \eqref{eq:momentarr}, \eqref{eq:uniarr}--\eqref{eq:lamdaconverge},  and  \eqref{eq:bar_item2}--\eqref{eq:bar_item5}, thereby proving Lemma~\ref{lem:residualUI}.

In the rest of this section, we prove \eqref{eq:bar_item1}--\eqref{eq:bar_item5} using the BAR \eqref{eq:hard_BAR}, which is repeated here as
\begin{align}\label{eq:lem_bar}
&\E_{\nu} \big[ {\cal A}f(X(0))\big] \\
& +  \frac{1}{t}\sum_{m=1}^{\infty} \E_{\nu}\Big[\big(f(X_{\delta_m})- f(X_{\delta_m-}) \big) 1(0 < T(\delta_m)  \leq t) \Big] = 0, \quad f \in \mathcal{D}, \notag 
\end{align}
where $\delta^{(n)}_m$, $T(\delta^{(n)}_m)$, $X^{(n)}_{\delta^{(n)}_m}$ and $X^{(n)}_{\delta^{(n)}_m-}$ are introduced below \eqref{eq:finite_events_assumption} in Section~\ref{sect:dynamics}, the operator $\cal{A}$ is defined in \eqref{eq:A}, and the class of functions $\cal{D}$ is defined above \eqref{eq:A}.

Fix an integer $\kappa \geq 1$ and $j \in \J$. We now prove \eqref{eq:bar_item1} by applying the BAR \eqref{eq:lem_bar} to $f_{\kappa,1}(x) = v_j \wedge \kappa$ (which belongs to $\cal{D}$) to obtain 
\begin{align*}
\E \big[1(&L^{(n)}_j(\infty)>0)1(R_{s,j}^{(n)}(\infty) < \kappa) \big]\\
 =&\  \frac{1}{t} \sum_{m=1}^{\infty} \E_{\nu^{(n)}}\Big[ \big(f_{\kappa,1}(X^{(n)}_{\delta^{(n)}_m})- f_{\kappa,1}(X^{(n)}_{\delta^{(n)}_m-}) \big) 1(0 < T(\delta^{(n)}_m)  \leq t) \Big].
\end{align*}
For station $j$, we recall the sequence  of service times $\{T_{s,j}^{(n)}(q),\ q \geq 1\}$ and the sequence of departure times  $\{S_j^{(n)}(q),\ q \geq 1\}$, defined in \eqref{eq:seq_service} and \eqref{eq:service_times}, respectively. From the form of the test function, we see that $f_{\kappa,1}(X^{(n)}_{\delta^{(n)}_m})- f_{\kappa,1}(X^{(n)}_{\delta^{(n)}_m-}) = 0$ for all events $\delta^{(n)}_m$ that do not correspond to departures from station $j$. Therefore,
\begin{align}\label{eq:service_jumps}
& \sum_{m=1}^{\infty} \E_{\nu^{(n)}}\Big[ \big(f_{\kappa,1}(X^{(n)}_{\delta^{(n)}_m})- f_{\kappa,1}(X^{(n)}_{\delta^{(n)}_m-}) \big) 1(0 < T(\delta^{(n)}_m)  \leq t) \Big]  \\
&\quad =   \sum_{q=1}^\infty \E_{\nu^{(n)}} \Big[f_{\kappa,1}\big(T_{s,j}^{(n)}(q)\big)1(S_j^{(n)}(q)\le t) \Big] \notag  \\
& \quad =   \sum_{q=1}^\infty \E\big[ f_{\kappa,1}(T_{s,j}^{(n)}) \big] \E_{\nu^{(n)}} \Big[1(S_j^{(n)}(q)\le t) \Big] \notag  \\
& \quad =   \E_{\nu^{(n)}} \big[ D_j^{(n)}(0,t] \big] \E\big[ f_{\kappa,1}(T_{s,j}^{(n)}) \big]. \notag
\end{align}
In the second equality, we used the independence of $T_{s,j}^{(n)}(q)$ and $S_j^{(n)}(q)$, and in the last equality we use the Fubini-Tonelli theorem, justified by the fact that the summands are non-negative. We therefore have
\begin{align*}
&\E \big[1(L_j^{(n)}(\infty)>0)1(R_{s,j}^{(n)}(\infty) < \kappa)\big]
\\
& \quad =   \frac{1}{t}\E_{\nu^{(n)}} \big[ D_j^{(n)}(0,t] \big] \E\big[ f_{\kappa,1}(T_{s,j}^{(n)}) \big], \qquad j \in \J.
\end{align*}
We set $t = 1$, take the limit as $\kappa \to \infty$ on each side, and apply the monotone convergence theorem to conclude the proof of \eqref{eq:bar_item1}. 

To prove \eqref{eq:bar_item2}, we use the test function $f_{\kappa,2}(x) = u_i\wedge \kappa$ (with $\kappa \geq 1$, and $i \in \sr{E}$) and repeat the argument for \eqref{eq:bar_item1}, but instead of $S_j^{(n)}(q)$, we use $U_i^{(n)}(q)$, as defined in \eqref{eq:arrival_times} to be the $q$th external arrival time to station $i \in \sr{E}$.  An alternative way to prove \eqref{eq:bar_item2} would be to verify that the process $E_i^{(n)}$ is stationary when $E_i^{(n)}(0)$ is initialized accoridng to $\nu^{(n)}$. Then \eqref{eq:bar_item2} would follow from \cite[Theorem 76]{Serf2009}.

We now move on to prove \eqref{eq:bar_item3}. We recall from \eqref{eq:dynamics_queue} that on every sample path,
\begin{align} \label{eq:sample_queue}
L_j^{(n)}(t) - L_j^{(n)}(0) = E_j^{(n)}(0,t] - D_j^{(n)}(0,t] + \sum_{k \in \J} \Phi^{(k)}_j(D_k^{(n)}(0,t]),
\end{align}
where $E_j(t) \equiv 0$ if $j \in \J \setminus \sr{E}$ and  $\{\Phi^{(k)}(m), m \geq 1 \} $ is the customer routing process from \eqref{eq:routing_vector}. We first show that 
\begin{align*}
\E_{\nu^{(n)}} \Big[L_j^{(n)}(t) - L_j^{(n)}(0) \Big] =0,
\end{align*}
which would have been immediate provided we knew that $\E \big[ L_j^{(n)}(\infty) \big] < \infty$. Instead, we use the stationarity of $\nu^{(n)}$ to see that for every $\kappa \geq 1$, 
\begin{align*}
\E_{\nu^{(n)}} \big[ (L_j^{(n)}(t)\wedge \kappa) - (L_j^{(n)}(0)\wedge \kappa)\big] = 0.
\end{align*}
We wish to take the limit as $\kappa \to \infty$ and apply the dominated convergence theorem. To justify doing so, observe that 
\begin{align*}
\big| (L_j^{(n)}(t)\wedge \kappa) - (L_j^{(n)}(0)\wedge \kappa)\big| \leq E_j^{(n)}(0,t] + \sum_{k \in \J} D_k^{(n)}(0,t],
\end{align*} 
and by \eqref{eq:finite_intensities}, the expectation of the right hand side (with respect to $\nu^{(n)}$) is finite. We apply the dominated convergence theorem to get
\begin{align*}
0 = \lim_{\kappa \to \infty} \E_{\nu^{(n)}} \big[ (L_j^{(n)}(t)\wedge \kappa) - (L_j^{(n)}(0)\wedge \kappa)\big] = \E_{\nu^{(n)}} \Big[L_j^{(n)}(t) - L_j^{(n)}(0) \Big],
\end{align*}
implying
\begin{align*}
&\ \E_{\nu^{(n)}} \big[E_j^{(n)}(0,t]\big] - \E_{\nu^{(n)}} \big[D_j^{(n)}(0,t]\big] + \sum_{k \in \J} \E_{\nu^{(n)}}\Big[\Phi^{(k)}_j(D_k^{(n)}(0,t]) \Big] = 0. 
\end{align*}
Suppose we knew that
\begin{align}
\E_{\nu^{(n)}}\Big[\Phi^{(k)}_j(D_k^{(n)}(0,t]) \Big] =&\ p_{kj} \E_{\nu^{(n)}} \big[ D_k^{(n)}(0,t]\big]. \label{eq:wts_bar3}
\end{align}
Then 
\begin{align*}
\E_{\nu^{(n)}} \big[E_j^{(n)}(0,t]\big] - \sum_{k \in \J} \E_{\nu^{(n)}} \big[ D_k^{(n)}(0,t]\big](\delta_{kj} - p_{kj}) = 0, \quad j \in \J.
\end{align*}
Letting $D^{(n)}(0,t] \in \R^d$ be the vector whose elements are $ D_k^{(n)}(0,t]$, we would use \eqref{eq:traffic 2} and \eqref{eq:bar_item2} to conclude that
\begin{align*}
\E \big[  D^{(n)}(0,t]\big] = (I - P^{\rs{T}})^{-1}\lambda_e^{(n)} t = \lambda_{a}^{(n)}t, 
\end{align*}
thereby completing the proof of  \eqref{eq:bar_item3}. It remains to verify \eqref{eq:wts_bar3}. Observe that 
\begin{align*}
\E_{\nu^{(n)}}\Big[\Phi^{(k)}_j(D_k^{(n)}(0,t]) \Big]  =&\ \E_{\nu^{(n)}}\bigg[\sum_{q=1}^{ \infty} \phi^{(k)}_j(q)1(S_j^{(n)}(q) \leq t) \bigg]\\
 =&\ \sum_{q=1}^{ \infty} \E_{\nu^{(n)}}\Big[\phi^{(k)}_j(q)1(S_j^{(n)}(q) \leq t) \Big],
\end{align*}
where $\phi^{(k)}_j$ is the indicator that the $q$th customer departing station $k$ routes to station $j$, as defined in \eqref{eq:routing_decisions}, and in the second equality we apply the Fubini-Tonnelli theorem, justified by the non-negativity of the summands. By definition, $\phi^{(k)}_j(q)$ is independent of $S_j^{(n)}(q)$ for all $q \geq 1$, and
\begin{align*}
\E_{\nu^{(n)}} \big[ \phi^{(k)}_j(q)\big] = p_{kj}, \quad q \geq 1.
\end{align*}
Repeating the arguments used to obtain \eqref{eq:service_jumps}, we see that
\begin{align*}
\E_{\nu^{(n)}}\Big[\Phi^{(k)}_j(D_k^{(n)}(0,t]) \Big] =&\ p_{kj} \E_{\nu^{(n)}} \big[ D_k^{(n)}(0,t]\big].
\end{align*} 
 This proves \eqref{eq:wts_bar3} and concludes the proof of \eqref{eq:bar_item3}.

We move on to prove \eqref{eq:bar_item4}. For $j \in \J$ and $K > 0$, we introduce the test function
\begin{align*}
f_{\kappa,3}(x) = [v_j^21(v_j>K) + K^21(v_j\leq K)] \wedge \kappa^2, \quad \kappa \geq \left\lceil K \right\rceil.
\end{align*}
Observe that
\begin{align*}
f_{\kappa,3}'(v_j) = 2 v_j 1(K < v_j < \kappa).
\end{align*}
Plugging $f_{\kappa,3}(x)$ into \eqref{eq:lem_bar} and repeating the arguments used to get \eqref{eq:service_jumps}, we get
\begin{align*}
&\E \big[ 1(L^{(n)}_j(\infty) > 0) R_{s,j}^{(n)}(\infty) 1(K < R_{s,j}^{(n)}(\infty) < \kappa) \big]\\
 & \quad=  \frac{1}{2}\E_{\nu^{(n)}} \big[ D_j^{(n)}(0,1] \big] \E\big[ f_{\kappa,3}(T_{s,j}^{(n)}) - K^2 \big].
\end{align*}
Taking $\kappa \to \infty$ and applying the monotone convergence theorem proves \eqref{eq:bar_item4}. Finally, for $i \in \sr{E}$, the argument for \eqref{eq:bar_item5} is identical once we use the test function
\begin{align*}
f_{\kappa,4}(x) = [u_i^21(u_i>K) + K^21(u_i\leq K)] \wedge \kappa^2, \quad \kappa \geq \left\lceil K \right\rceil.
\end{align*}
This concludes the proofs of Lemmas~\ref{lem:boundary_probs} and \ref{lem:residualUI}.

\subsection{Proof of \lem{limiting 1}}
\label{app:limiting 1}
We begin with \eq{limiting 1}. Recall the definition of $\gamma(\theta)$ from \eqref{eq:gamma 1} and for $i \in \sr{E}$ and $j \in \J$, the quadratic approximations of $\eta_i^{(n)}(\theta_i)$ and $\zeta_j^{(n)}(\theta)$ from \eqref{eq:eta asymptotic 1} and \eqref{eq:zeta asymptotic 1}. Lemma~\ref{lem:uniform 1} tells us that the error from these quadratic approximations vanishes. To conclude \eqref{eq:limiting 1}, it remains to show that for all $i \in \sr{E}$ and $j \in \J$,
\begin{align*}
&\lim_{n \to \infty} \frac{1}{r_n} \abs{\widetilde{\lambda}_{e,i}^{(n)} - {\lambda}_{e,i}^{(n)} } = 0, \quad \lim_{n \to \infty} \frac{1}{r_n} \abs{\widetilde{\lambda}_{s,j}^{(n)} - {\lambda}_{s,j}^{(n)} } = 0, \\
&\lim_{n \to \infty}  \abs{\widetilde{\sigma}_{e,i}^{(n)} - {\sigma}_{e,i}^{(n)}} = 0, \quad \lim_{n \to \infty}  \abs{\widetilde{\sigma}_{s,j}^{(n)} - {\sigma}_{s,j}^{(n)}} = 0.
\end{align*}
The latter two statements are a consequence of \eqref{eq:momentarr}, \eqref{eq:momentser}, and \eqref{eq:tildes_converge}. We observe that
\begin{align*}
\frac{1}{r_n} \abs{\widetilde{\lambda}_{e,i}^{(n)} - {\lambda}_{e,i}^{(n)} } =&\ \widetilde{\lambda}_{e,i}^{(n)}{\lambda}_{e,i}^{(n)}\frac{1}{r_n}  \big|\E T^{(n)}_{e,i} - \E (T^{(n)}_{e,i}\wedge 1/r_n) \big| \\
=&\ \widetilde{\lambda}_{e,i}^{(n)}{\lambda}_{e,i}^{(n)}\frac{1}{r_n} \E (T^{(n)}_{e,i} 1(T^{(n)}_{e,i}>1/r_n)) \\
\leq &\ \widetilde{\lambda}_{e,i}^{(n)}{\lambda}_{e,i}^{(n)} \big| \E ((T^{(n)}_{e,i})^2 1( T^{(n)}_{e,i}>1/r_n))\big| \to 0 \quad \text{ as $n \to \infty$},
\end{align*}
where the convergence to zero is by \eqref{eq:momentarr}, \eqref{eq:tildes_converge} and the uniform integrability of $\{(T^{(n)}_{e,i})^2, n \geq 1\}$. The same argument holds for $\widetilde{\lambda}_{s,j}^{(n)}$, proving \eq{limiting 1}. 

We move on to verify \eqref{eq:limiting 2}, or
\begin{align}
\lim_{n \to \infty} \sup_{\substack{\theta < 0 \\ \|\theta\| \leq K}} \Big| \E \big[f_{r_n \theta}^{(n)}(X^{(n)}(\infty))\big]  - \varphi^{(n)}(\theta) \Big| = 0. \label{eq:non_limiting 2}
\end{align}
Recalling the definitions of $f_{r_n \theta}^{(n)}$ and $\varphi^{(n)}(\theta)$ from \eqref{eq:testf} and \eqref{eq:laplaceprelimit}, we have
\begin{align}\label{eq:non_limit2_interm}
& \sup_{\substack{\theta < 0 \\ \|\theta\| \leq K}}  \Big| \E \big[f_{r_n \theta}^{(n)}(X^{(n)}(\infty))- e^{\br{\theta, {Z}^{(n)}(\infty)}}\big] \Big| \\
&\quad \leq \ \sup_{\substack{\theta < 0 \\  \|\theta\| \leq K}}   \E \Big|e^{\sum_{i \in \sr{E}} \eta_i^{(n)}(r_n\theta_i)g^{(n)}(R^{(n)}_{e,i}(\infty))+\sum_{j \in \J} \zeta_j^{(n)}(r_n\theta) g^{(n)}(R^{(n)}_{s,j}(\infty)) }- 1\Big|  \notag \\
&\quad\leq \ c_f(K)  \sup_{\substack{\theta < 0 \\ \|\theta\| \leq K}} \E \Big{|}  \sum_{i \in \sr{E}} (R^{(n)}_{e,i}(\infty)\wedge 1/r_n)\eta_i^{(n)}(r_n\theta_i)\notag \\
&\qquad \qquad +\sum_{j \in \J}(R^{(n)}_{s,j}(\infty) \wedge 1/r_n)\zeta_j^{(n)}(r_n\theta_i)\Big{|}.  \notag
\end{align}
To obtain the last inequality, we used \eqref{eq:exp_ineq} and  
\begin{align*}
\sup_{\|\theta\| \leq K}e^{|\sum_{i \in \sr{E}} \eta_i^{(n)}(r_n\theta_i)g^{(n)}(R^{(n)}_{e,i}(\infty))+\sum_{j \in \J} \zeta_j^{(n)}(r_n\theta) g^{(n)}(R^{(n)}_{s,j}(\infty)) |} \leq c_f(K),
\end{align*}
where $c_f(K)$ is defined in \eqref{eq:testf_bound}. For $i \in \sr{E}$ we have
\begin{align*}
&\lim_{n\to \infty} \sup_{\substack{\theta < 0 \\ \|\theta\| \leq K}} \E\Big{|} (R^{(n)}_{e,i}(\infty)\wedge 1/r_n)\eta_i^{(n)}(r_n\theta_i)\Big{|}\\
&\quad\leq\ \lim_{n\to \infty} \hat{c}_{e,i}^{(n)}(r_nK) r_nK \E\big[ R^{(n)}_{e,i}(\infty)\big] = 0,
\end{align*}
where the inequality is by \eqref{eq:eta asymptotic 0} and the fact that the limit equals zero is by Lemma~\ref{lem:uniform_slopes} and Lemma~\ref{lem:residualUI}. A similar argument holds for $R_{s,j}^{(n)}(\infty)$, implying \eqref{eq:non_limiting 2}. 

We move on to verify \eqref{eq:limiting 3}, or 
\begin{align}
&\lim_{n \to \infty}  \sup_{\substack{\theta < 0 \\  \|\theta\| \leq K}} \frac{1}{r_n} \Big| \E \big[1( L_{k}^{(n)}(\infty)=0)\big(f_{r_n \theta}^{(n)}(X^{(n)}(\infty))- e^{\br{\theta, {Z}^{(n)}(\infty)}}\big)\big] \Big| \label{eq:non_limiting 3} \\
&\quad= 0, \quad k \in \J. \notag
\end{align}
Repeating the steps used to obtain \eqref{eq:non_limit2_interm}, we have
\begin{align}\label{eq:non_limit3_interm}
& \sup_{\substack{\theta < 0 \\ \|\theta\| \leq K}} \frac{1}{r_n} \Big| \E \big[1( L_{k}^{(n)}(\infty)=0)\big(f_{r_n \theta}^{(n)}(X^{(n)}(\infty))- e^{\br{\theta, {Z}^{(n)}(\infty)}}\big)\big] \Big|  \\
& \quad\leq \ c_f(K)\sum_{i \in \sr{E}}\E \Big|  K\hat{c}_{e,i}^{(n)}(K)  1( L_{k}^{(n)}(\infty)=0)(R^{(n)}_{e,i}(\infty)\wedge 1/r_n)\Big| \notag \\
& \qquad + c_f(K)\sum_{j \in \J}\E \Big|K\hat{c}_{s,j}^{(n)}(r_nK)  1( L_{k}^{(n)}(\infty)=0)(R^{(n)}_{s,j}(\infty)\wedge 1/r_n)\Big|. \notag
\end{align}
To show that \eqref{eq:non_limit3_interm} vanishes as $n \to \infty$, observe that by \eqref{eq:uniarr} and Lemma~\ref{lem:boundary_probs}, 
\begin{align*}
&\E \big[  1( L_{k}^{(n)}(\infty)=0)(R^{(n)}_{e,i}(\infty)\wedge 1/r_n)\big]\\
&\quad = \ \E \big[  1( L_{k}^{(n)}(\infty)=0)(R^{(n)}_{e,i}(\infty)\wedge 1/r_n)\big(1(R^{(n)}_{e,i}(\infty) > r_n^{-1/2}) \\
&\qquad \qquad + 1(R^{(n)}_{e,i}(\infty) \leq r_n^{-1/2}) \big) \big]\\
& \quad \leq \ \E \big[ R^{(n)}_{e,i}(\infty)1(R^{(n)}_{e,i}(\infty) > r_n^{-1/2})\big] + \E \big[  1( L_{k}^{(n)}(\infty)=0)r_n^{-1/2} \big] \\
& \quad \to 0 \quad \text{ as $n \to \infty$}.
\end{align*}

Repeating similar arguments for $R^{(n)}_{s,j}(\infty)$ proves \eqref{eq:non_limiting 3}.

\section{Proof of Lemma~\ref{LEM:LAPLACETIGHTNESS}}
\label{app:laplacetightness}
\setnewcounter
We begin by proving \eqref{eq:iclaimtightness}. Fix $M>0$ and $\delta > 0$. Setting 
\begin{align*}
\theta = -\delta(1, 1, ... , 1)^{\rs{T}} = -\delta \vc{1}, 
\end{align*}
we see that 
\begin{align} \label{eq:integralsplit}
&\phi^{(n)}(\theta) = \phi^{(n)}(\delta) \\
&= \int_{\R_+^d} e^{-\delta\br{\vc{1}, x}} \nu^{(n)}(dx)  \notag \\
&=  \int_{\{||x||_{\infty} \leq M\}} e^{-\delta \br{\vc{1}, x}} \nu^{(n)}(dx) + \int_{\{||x||_{\infty} \leq M\}^{c}} e^{-\delta\br{\vc{1}, x}} \nu^{(n)}(dx).  \notag
\end{align}
Hence, we have 
\begin{align*}
\phi^{(n)}(\theta) \geq e^{-\delta M d} \nu^{(n)}\big{\{}||x||_{\infty} \leq M\big{\}}.
\end{align*}
Assuming that $\{\nu^{(n)}\}$ is tight, we can fix $\epsilon > 0$ and take $M$ large enough and then $\delta$ small enough such that 
\begin{align*}
\phi(\delta) =  \lim \limits_{n \to \infty} \phi^{(n)}(\delta) \geq 1-\epsilon.
\end{align*}
Since each $\phi^{(n)}$ is monotone, $\phi$ is monotone too, which immediately implies 
\begin{align*}
\lim \limits_{\theta \uparrow 0} \phi(\theta) = 1.
\end{align*}
Furthermore, \eqref{eq:integralsplit} also implies that
\begin{align*}
\phi^{(n)}(\delta) \leq \nu^{(n)}\big{\{}||x||_{\infty} \leq M\big{\}} + e^{-\delta M},
\end{align*}
or 
\begin{align*}
\phi(\delta) -  e^{-\delta M} \leq \liminf \limits_{n \to \infty} \nu^{(n)}\big{\{}||x||_{\infty} \leq M\big{\}}.
\end{align*}
If $\lim_{\theta \uparrow 0} \phi(\theta) = 1$, we can fix $\epsilon > 0 $ and choose a $\delta$ small enough so that $\phi(\delta) > 1 - \epsilon/2$. Then take $M$ large enough so that $e^{-\delta M} < \epsilon/2$ to conclude that
\begin{align*}
\liminf \limits_{n \to \infty} \nu^{(n)}\big{\{}||x||_{\infty} \leq M\big{\}} \geq 1 - \epsilon,
\end{align*}
which proves \eqref{eq:iclaimtightness}. \dgreen{Statement \eqref{eq:iiclaimtightness} is an immediate consequence of \cite[Theorem 5.22]{Kall2001}.}

\bibliography{dai20170409}

\end{document}